\documentclass[10pt,a4paper]{amsart}
\usepackage{fourier} 
\usepackage{amsthm}
\usepackage{xargs}                 
\usepackage[pdftex,dvipsnames]{xcolor} 

\usepackage{amsmath}
\usepackage{braket}
\usepackage[colorlinks,citecolor=blue,backref=page]{hyperref}
\usepackage{multirow}

\hypersetup{
  pdftitle   = {},
  pdfauthor  = {},
  pdfcreator = {\LaTeX\ with package \flqq hyperref\frqq}
}
\usepackage{longtable}
\usepackage{pdflscape}

\usepackage{amssymb}
\usepackage{pifont}

\newtheorem{theorem}{Theorem}[section]
\newtheorem*{theorem*}{Theorem}
\newtheorem{theorem-non}{Theorem}
\newtheorem{proposition}[theorem]{Proposition}
\newtheorem{lema}[theorem]{Lemma}
\newtheorem*{lema*}{Lemma}

\newtheorem*{conjecture*}{Conjecture}

\theoremstyle{definition}
\newtheorem{definition}[theorem]{Definition}

\theoremstyle{remark}
\newtheorem{obs}[theorem]{Remark}

\newtheorem{question}[theorem]{Question}

\DeclareMathOperator{\ad}{ad}

\DeclareMathOperator{\tr}{tr}

\DeclareMathOperator{\Ad}{Ad}
\DeclareMathOperator{\id}{id}

\DeclareMathOperator{\Span}{span}
\DeclareMathOperator{\Lie}{Lie}

\DeclareMathOperator{\htt}{ht}
\DeclareMathOperator{\vol}{vol}

\newcommand{\got}{\mathfrak}

\numberwithin{equation}{section}

\addtocontents{toc}{\setcounter{tocdepth}{2}}

\usepackage[colorinlistoftodos,prependcaption,textsize=tiny]{todonotes}
\newcommandx{\duvida}[2][1=]{\todo[linecolor=red,backgroundcolor=red!25,bordercolor=red,#1]{#2}}
\newcommandx{\completar}[2][1=]{\todo[linecolor=blue,backgroundcolor=blue!25,bordercolor=blue,#1]{#2}}
\newcommandx{\info}[2][1=]{\todo[linecolor=OliveGreen,backgroundcolor=OliveGreen!25,bordercolor=OliveGreen,#1]{#2}}
\newcommandx{\melhorar}[2][1=]{\todo[linecolor=Plum,backgroundcolor=Plum!25,bordercolor=Plum,#1]{#2}}
\newcommandx{\apagar}[2][1=]{\todo[disable,#1]{#2}}

\title[Scalar curvatures of invariant almost Hermitian structures]{Scalar curvatures of invariant almost Hermitian structures on flag manifolds with two and three isotropic summands}
\author[]{Lino Grama}
\address{Imecc - Unicamp, Departamento de Matem\'{a}tica. Rua S\'{e}rgio Buarque de Holanda,
651, Cidade Universit\'{a}ria Zeferino Vaz. 13083-859 Campinas - SP, Brazil}
\email{linograma@gmail.com}

\author[]{Ailton R. Oliveira}
\address{UEMS - Universidade Estadual de Mato Grosso do Sul - MS, Cidade Universitária de Dourados, Rodovia Itahum, Km 12 s/n - Jardim Aeroporto, Dourados - MS, Brazil}
\email{ailton\_rol@yahoo.com.br}

\date{\today}

\begin{document}

\maketitle
\begin{abstract} 
In this paper we study invariant almost Hermitian geometry on generalized flag manifolds which the isotropy representation decompose into two or three irreducible components. We will provide a classification of such flag manifolds admitting {\em K\"ahler like scalar curvature metric}, that is, almost Hermitian structures $(g,J)$ satisfying $s=2s_C$ where $s$ is Riemannian scalar curvature and $s_C$ is the Chern scalar curvature.
\end{abstract} 

\section{Introduction}

In the present paper we will study curvature properties of invariant almost Hermitian structures on homogeneous spaces. We search for {\em K\"ahler like scalar curvature {\rm (Klsc)} metrics}, that is, almost Hermitian structures $(g,J)$ satisfying $s=2s_C$ where $s$ is the Riemannian scalar curvature and $s_C$ is the Chern scalar curvature on generalized flag manifolds with two or three isotropy summands.

Recall the geometric meaning of the Riemannian scalar curvature $s$: given a Riemannian manifold $(M,g)$, the volume of the geodesic ball of radius $r$ with center at $p\in M$ has the asymptotic expansion as follows
\begin{equation}\label{ee1}
\vol(B(p,r))=\omega_{n}r^{n}\left(1-\dfrac{s(p)}{6(n+2)}r^{2}+\mathcal{O}(r^{4})\right),
\end{equation}
where $\omega_{n}$  is the volume of the unity ball in  $\mathbb{R}^{n}$, see \cite{besse}, \cite{lock}. Therefore  the Riemannian scalar curvature $s(p)$ is positive or negative at a point $p$, if the volume of a small geodesic ball at $p$ is respectively smaller or larger than the corresponding Euclidean ball of the same radius. It is well known that if $(g,J)$ is K\"ahler, then $s=2s_C$ (see \cite{moroianu}). On the other hand, if the metric is non-K\"ahler, there is no immediate geometric interpretation for the Chern scalar curvature $s_C$. However if the Hermitian structure is non-K\"ahler but satisfies $s=2s_C$, has the standard representation such as the one provided by equation (\ref{ee1}). Another motivation is that the existence of {\em Klsc metric} allow us to ask  about a Kazdan--Warner--type theorem (see \cite{Kaz-War}) for admissible functions of scalar curvature, in the case of Chern scalar curvature, such as in \cite{Fusi}. For results on geometric flows related to complex manifolds (e.g. Chern-Ricci flow) we suggest \cite{TZ}.

Dabkowski and Lock in \cite{lock} call the Hermitian metrics satisfying 
\begin{equation}\label{ee1}
s=2s_C
\end{equation}
by {\it K\"ahler like scalar curvature (Klsc) metric}, and the authors exhibit examples of non-compact Hermitian manifolds satisfying the equation \ref{ee1}.

According to \cite{integra} and \cite{apo} nearly-K\"ahler manifolds and almost-K\"ahler manifolds satisfying $s=2s_{C}$ are K\"ahler. Fu and Zhou showed in \cite{fu} that, if the pair $(g,J)$ belongs to the Gray-Hervella class $\mathcal{W}_{2}\oplus \mathcal{W}_{3}\oplus \mathcal{W}_{4}$ with $s=2s_{c}$, then $(g,J)$ is K\"ahler.
 Recently, Lejmi-Upmeier proposed the following Question (Remark 3.3 of \cite{integra}):
\begin{question} \label{question-princ}
Do higher-dimensional closed almost Hermitian non-K\"ahler manifolds with $2s_{C}=s$ exist?
\end{question}

On the other hand, let us denote by $\nabla^t$ be the 1-parameter family of connections introduced by Gauduchon in \cite{gau}: \begin{eqnarray*}
g(\nabla^{t}_{X}Z,Y)&=&g(D_{X}Z,Y)-\frac{1}{2}g(J(D_{X}J)Z,Y)\\
&+&\frac{t}{4}g((D_{JY}J+JD_{Y}J)X,Z)-\frac{t}{4}g((D_{JZ}J+JD_{Z}J)X,Y),
\end{eqnarray*}
where $D$ is the Levi-Civita connection.
 It is well know that $\nabla^1$ coincides with the Chern connection.
 According to \cite{fu}, let us define the Hermitian scalar curvatures $$s_{1}(t)=R^{t}(u_{\overline{i}},u_{i},u_{j},u_{\overline{j}})$$
and $$s_{2}(t)=R^{t}(u_{\overline{i}},u_{j},u_{i},u_{\overline{j}}),$$
where  $R^{t}$ is the curvature tensor associated to $\nabla^{t}$ and $\{u_{i}\}_{i=1,2,\cdots,n}$ is a unitary frame. We will call $s_1(t)$ and $s_2(t)$ the first and second Hermitian scalar curvatures. One of the main features of the generalized flag manifolds  relies in the fact that the scalar curvature $s_1(t)$ is constant. This manner, thereon we shall denote it solely by $s_1$. Furthermore, it coincides with the Chern scalar curvature $s_C$. See Section \ref{scal-hermit} for details about the comments above.  

With the discussion above in mind, let us rephrase the Question \ref{question-princ} in our context:
\begin{question} \label{question-princ-flag}
Do higher-dimensional generalized flag manifolds $G/K$ equipped with $G$-invariant almost Hermitian structure  with $2s_{1}=s$ exist?
\end{question} 

In our previous \cite{nosso2} we study the Question \ref{question-princ-flag} in a class of {\em generalized flag manifolds} equipped with an invariant almost Hermitian structure. In this work we provided a complete answer to Question \ref{question-princ-flag} in the context of invariant Hermitian structures for a large class of homogeneous space, namely flag manifolds with two and three isotropy summands. As we have done in \cite{nosso2}, we proceed by computing explicitly the Hermitian scalar curvatures using the tools of Lie Theory. We recall the description of these manifolds in Sections \ref{sec:flag2somandos} and \ref{sec:flag3somandos}.

We obtain the following results (for a description of invariant metrics and almost complex structures on flag manifolds, see Section \ref{sec:prelim}):

\begin{theorem}[= Theorem \ref{teo-2somandos}]
Let $G/K$ be a generalized flag manifold with two isotropy summands, and let $g=(\lambda_1,\lambda_2)$ be an invariant Riemannian metric. Then $2s_1-s=0$ holds for $G/K$ if the pair $(g,J_i)$ satisfies the following conditions: 

\begin{itemize}
    \item For the invariant complex struture $J=(+,+)$, the metric $g$ is parameterized by $\lambda_2=2\lambda_1$. In this case, $(g,J)$ is K\"ahler.
    \item For the invariant almost complex structure $J_1=(+,-)$, the metric $g$ is parameterized by $\lambda_2=2\lambda_1( \sqrt{10}+3)$. In this case, $(g, J_1)\in\mathcal{W}_1\oplus \mathcal{W}_2$. 
\end{itemize}
\end{theorem}

The analysis of the equation $2s_1-s=0$ in the case of flag manifolds with three isotropy summands is a little bit trickier, so we deal with the question working case-by-case. The description of the flag manifolds with three summands splits into two sub-families: {\em Type I} listed in Table \ref{table:3comp-tipo1} and {\em Type II} listed in Table \ref{table:3comp-tipo2}. Theorem \ref{teoB-intro} below encodes our results in a simple fashion. The details and explicit computations can be found in Section \ref{sec:flag3somandos}. It is worth to point out that when $(g,J)$ is K\"ahler, the equation $2s_1-s=0$ is satisfied. In this case, by {\em non-trivial solution of $2s_1-s=0$} we mean that the pair $(g,J)$ is non-K\"ahler. 

\begin{theorem}\label{teoB-intro}
Let $G/K$ be a generalized flag manifold with three isotropy summands, equipped with an invariant almost complex structure (IACS) $J_i$, $i=1,\ldots,4$. Then there exists an invariant Riemannian metric $g$ such that the almost Hermitian structure $(g,J_i)$ admits a non-trivial solution of the equation $2s_1-s=0$ in the cases listed in the Table \ref{tab:3somandos-into}. Moreover, for each IACS admiting non-trivial solution we provide explicity families of invariant K\"ahler-like scalar metrics $g$ on $G/K$. 


\begin{table}[h!] 
\caption{K\"ahler-like scalar curvature metric on flag manifolds with 3 isotropy summands.} \label{tab:3somandos-into}
\begin{tabular}{|l|l|c|c|l|}
\hline
IACS                   & \multicolumn{1}{c|}{\begin{tabular}[c]{@{}c@{}}Flag manifold \\ $G/K$\end{tabular}} & \multicolumn{1}{l|}{Integrable ?} & \begin{tabular}[c]{@{}c@{}}non-trivial solution\\ $2s_1-s=0$\end{tabular} & \multicolumn{1}{c|}{Details}                    \\ \hline
\multirow{2}{*}{$J_1=(+,+,+)$} & type I                                                                              & \ding{51}                       & \ding{55}                                                               & \multicolumn{1}{c|}{{\rm Thm} $\ref{teo3somandosJ1}$} \\ \cline{2-5} 
                       & type II                                                                             & \ding{51}                       & \ding{55}                                                               & \multicolumn{1}{c|}{{\rm Thm} $\ref{teo3somandosJ1}$} \\ \hline
\multirow{2}{*}{$J_2=(-,+,+)$} & type I                                                                              & \ding{55}                       & \ding{51}                                                               & {\rm Thm} $\ref{teo3somandosJ2}$                      \\ \cline{2-5} 
                       & type II                                                                             & \ding{51}                       & \ding{55}                                                               & {\rm Thm} $\ref{teo3somandosJ2}$                      \\ \hline
\multirow{2}{*}{$J_3=(+,-,+)$} & type I                                                                              & \ding{55}                       & \ding{51}                                                               & {\rm Thm} $\ref{teo3somandosJ3}$                      \\ \cline{2-5} 
                       & type II                                                                             & \ding{51}                       & \ding{55}                                                               & {\rm Thm} $\ref{teo3somandosJ3}$                      \\ \hline
\multirow{2}{*}{$J_4=(+,+,-)$} & type I                                                                              & \ding{55}                       & \ding{51}                                                               & {\rm Thm} $\ref{teo3somandosJ4}$                      \\ \cline{2-5} 
                       & type II                                                                             & \ding{55}                       & \ding{51}                                                               & {\rm Thm} $\ref{teo3somandosJ4}$                      \\ \hline
\end{tabular}
\end{table}

\end{theorem}

In the last Section of the paper we work out three examples of flag manifolds with three isotropy summands: $G_{2}/U(2)$,  $SU(n+2)/S(U(n)\times U(1) \times U(1))$ and $F_{4}/SU(3)\times SU(2)\times U(1)$, providing family of example of {K\"ahler-like scalar curvature metrics} by using the techniques we have developed in this work.

\

{\em Acknowledgments: } We would like to thank Leonardo Cavenaghi and Eder Correa for for useful discussions. LG is partially supported by FAPESP grants 2018/13481-0, 2021/04003-0, 2021/04065-6  and CNPq grant no. 305036/2019-0.

\section{Invariant geometry of flag manifolds}\label{sec:prelim}

In this section we recall some well known results about the almost Hermitian geometry of generalized flag manifolds, from a Lie theoretical point of view. 
\subsection{Generalized flag manifolds}
Let $\got{g}^\mathbb{C}$ be a complex semi-simple Lie algebra. Given a Cartan subalgebra $\got{h}$ of $\got{g}^\mathbb{C}$, denotes by $\Pi$ the set of roots with respect to the pair $(\got{g}^\mathbb{C},\got{h})$. We have the following root space decomposition
$$
\got{g}^\mathbb{C}=\got{h}\oplus\sum_{\alpha\in \Pi}\got{g}_{\alpha},
$$ 
where $\got{g}_{\alpha}=\{X\in\got{g}:\forall H\in \got{h},[H,X]=\alpha(H)X \}$ denote the corresponding $1$-dimensional (complex) root space. 

The Cartan-Killing form is defined by  
$$
\langle X,Y\rangle=\tr(\ad(X)\ad(Y))
$$ 
and its restriction to $\got{h}$ is non-degenerate. Given a root $\alpha\in \got{h}^{*}$ we define $H_{\alpha}$ by $\alpha(\cdot)=\langle H_{\alpha},\cdot\rangle$. Moreover we denote $\got{h}_{\mathbb{R}}=\Span_{\mathbb{R}}\{H_{\alpha}:\alpha\in \Pi\}$ and $\got{h}_{\mathbb{R}}^{*}$ being the real subspace of $\got{h}^{*}$ spanned by the roots.

We will fix a Weyl basis of $\got{g}^\mathbb{C}$ as follows:
$$
X_{\alpha}\in\got{g}_{\alpha}, \ \mbox{ such that } \ \langle X_{\alpha},X_{-\alpha}\rangle=1 \ \mbox{and} \ [X_{\alpha},X_{\beta}]=m_{\alpha,\beta}X_{\alpha+\beta},
$$ 
with $m_{\alpha,\beta}\in\mathbb{R}$, $m_{-\alpha,-\beta}=-m_{\alpha,\beta}$ and $m_{\alpha,\beta}=0$ if $\alpha+\beta$ is not a root. 

Let $\Pi^{+}\subset \Pi$ be a choice of positive roots, $\Sigma$ be the corresponding system of simple roots and $\Theta$ be a subset of $\Sigma$. We will use the following notation: $\langle\Theta\rangle$ is the set of roots spanned by $\Theta$, $\Pi_{M}=\Pi\setminus\langle{\Theta}\rangle$ be the set of complementary roots and $\Pi_{M}^{+}$ be the set of complementary positive roots.

Let 
$$
\got{p}_{\Theta}=\got{h}\oplus
\sum_{\alpha\in\langle\Theta\rangle^{+}}\got{g}_{\alpha}\oplus
\sum_{\alpha\in\langle\Theta\rangle^{+}}\got{g}_{-\alpha}\oplus
\sum_{\beta\in \Pi_{M}^{+}}\got{g}_{\beta}
$$
be a parabolic sub-algebra of $\got{g}^\mathbb{C}$ determined by $\Theta$.

The {\em generalized flag manifold} $\mathbb{F}_{\Theta}$ is the homogeneous space 
$$
\mathbb{F}_{\Theta}=G^\mathbb{C}/P_{\Theta},
$$
where $G^\mathbb{C}$ is a complex connected Lie group with Lie algebra $\got{g}^\mathbb{C}$ and $P_{\Theta}$ is the normalizer of $\got{p}_{\Theta}$ in $G^\mathbb{C}$.

Let $\got{g}$ be the compact real form of $\got{g}^\mathbb{C}$. 
We have
$$
\got{g}=\Span_{\mathbb{R}}\{i\got{h}_{\mathbb{R}},A_{\alpha},iS_{\alpha}; \alpha\in \Pi\},
$$ 
where $A_{\alpha}=X_{\alpha}-X_{-\alpha}$ and $S_{\alpha}=X_{\alpha}+X_{-\alpha}$. We remark that the Lie algebra $\got{g}$ is semi-simple.

Denote by $G$ the compact real form of $G^\mathbb{C}$ with $\Lie(G)=\got{g}$ and let $\got{k}_{\Theta}$ the Lie algebra of  $K_{\Theta}:=P_{\Theta}\cap G$. It is well known that the Lie group $K_{\Theta}\subset G$ is a centralizer of a torus. 

We have 
$$
\got{k}_{\Theta}^{\mathbb{C}}=\got{h}\oplus
\sum_{\alpha\in\langle\Theta\rangle^{+}}\got{g}_{\alpha}\oplus
\sum_{\alpha\in\langle\Theta\rangle^{+}}\got{g}_{-\alpha},
$$
where $\got{k}_{\Theta}^{\mathbb{C}}$ denotes the complexification of the real Lie algebra $\got{k}_{\Theta}=\got{g}\cap \got{p}_{\Theta}$. 

The Lie group $G$ also acts transitively on $\mathbb{F}_{\Theta}$ and we have 
$$
\mathbb{F}_{\Theta}=G^\mathbb{C}/P_{\Theta}=G/(P_{\Theta}\cap G)=G/K_{\Theta}.
$$

When $\Theta=\emptyset$ the parabolic sub-algebra is a Borel sub-algebra of $\got{g}^\mathbb{C}$ and $\mathbb{F}=G^\mathbb{C}/P=G/T$ is called {\it full flag manifold}, where $T=P\cap G$ is a maximal torus of $G$. 
\bigskip

Recall that the flag manifold $\mathbb{F}_{\Theta}=G/K_{\Theta}$ is a reductive homogeneous space, that is, there exists a subspace $\got{m}$ of $\got{g}$ such that 
$$
\got{g}=\got{k}_{\Theta}\oplus \got{m} \ \ \mbox{and} \ \ 
\Ad(k)\got{m}\subseteq \got{m}, \forall k\in K_{\Theta},
$$
and one can identify the tangent space $T_{x_{0}}\mathbb{F}_{\Theta}$ with $\got{m}$, where $x_{0}=eK_{\Theta}$ is the origin of the $\mathbb{F}_{\Theta}$ (trivial coset).

Let us give a description of the tangent space $\got{m}$ in terms of the Lie algebra structure of $\got{g}$ as follows:
$$
\got{g}=\got{k}_{\Theta}\oplus\sum_{\beta\in \Pi_{M}}{\got{u}_{\beta}},
$$
with $\got{u}_{\beta}=\got{g}\cap(\got{g}_{\beta}\oplus \got{g}_{-\beta})$ and for each root $\beta \in \Pi_{M}$, $\got{u}_{\beta}$ has real dimension two and it is spanned by $A_{\beta}$ and $\sqrt{-1}S_{\beta}$, and we have the following identification $\got{m}=\sum_{\beta\in \Pi_{M}}\got{u}_{\beta}$.

\

An important tool to study invariant tensors on homogeneous space is the {\em isotropy representation}. We will restrict ourselves to the situation of flag manifolds. In this case, since the flag manifolds are {\em reductive} homogeneous spaces it is well know that the isotropy representation is equivalent to the following representation
$$\Ad(k)\mid_{\got{m}}:\got{m}\longrightarrow\got{m}.$$

The isotropy representation decomposes $\got{m}$ into irreducible components, that is, 
$$
\got{m}=\got{m}_{1}\oplus\got{m}_{2}\oplus\cdots\oplus\got{m}_{n},
$$
where each component $\got{m}_{i}$ satisfies $\Ad(K_{\Theta})(\got{m}_{i})\subset \got{m}_{i}$. 
Moreover, each component $\got{m}_{i}$ is irreducible, that is, the only invariant sub-spaces of $\got{m}_{i}$ by $\Ad(K_{\Theta})|_{\got{m}_{i}}$ are the trivial sub-spaces. We will call the sub-spaces $\got{m}_{i}$ by {\it isotropic summands} of the isotropy representation. 

\begin{obs}
From now on we will omit the symbol $\Theta$ whenever there is no risk of confusion. We will denote a flag manifold just by $\mathbb{F}=G/K$.
\end{obs}

\subsection{Invariant Metrics} 
Let us denote by $(\cdot,\cdot)$ the Cartan-Killing form of $\got{g}$. For each $\Ad(K)-$invariant inner product $(\cdot ,\cdot)_{\Lambda}$ on $\mathfrak{m}$, there exists a unique $(\cdot,\cdot)-$self-adjoint, positive operator $\Lambda :\mathfrak{m}\longrightarrow\mathfrak{m}$ commuting with $\Ad(k)\left|_{\mathfrak{m}}\right.$ for all $k\in K$ such that 
$$(X,Y)_{\Lambda}=(\Lambda X,Y), \,\, X,Y\in\mathfrak{m}.$$

Therefore an invariant Riemannian metric $g$ on $\mathbb{F}$ is completely determined by the invariant inner product $(X,Y)_{\Lambda}$, and the inner product is determined by $\Lambda$. 

The vectors $A_{\alpha}$, $\sqrt{-1}S_{\alpha}$, $\alpha\in\Pi$,  are eigenvectors of $\Lambda$ associated to the same eigenvalue $\lambda_\alpha$. 

The invariant inner product $(X,Y)_{\Lambda}$ admits a natural extension to a symmetric bi-linear form on  $\got{m}^{\mathbb{C}}$. On the complexified tangent space we have $\Lambda(X_{\alpha})=\lambda_{\alpha}X_{\alpha}$ with  $\lambda_{\alpha}>0$ and $\lambda_{-\alpha}=\lambda_{\alpha}$. 

\

{\bf Notation:} In the sequence of this work, we will abuse notation and will denote the invariant metric $g$ just by the operator $\Lambda$ associated to the invariant inner product. We also denote the invariant metric $g$ just by a $n$-tuple of positive numbers $(\lambda_1,\ldots, \lambda_n)$ representing the eigenvalues of the operator $\Lambda$ and is parameterized by the number of irreducible components.   \qed

\subsection{Invariant Almost Complex Structures}
\begin{definition}
An almost complex structure on the flag manifold $\mathbb{F}$ is a tensor $J$ such that for every point $x\in \mathbb{F}$, there is an endomorphism  $J:T_{x}\mathbb{F}\longrightarrow T_{x}\mathbb{F}$ such that $J^{2}=-\id$.
\end{definition}

\begin{definition}
A $G$-invariant almost complex structure $J$ on $\mathbb{F}=G/K$ is an almost complex structure that satisfies  
$$
J_{ux}=dE_{u}J_{x}dE_{u^{-1}},\ \mbox{for all } \ u\in G,
$$ 
where $dE_{u}:T(G/K) \to T(G/K)$ denotes the differential of the left translation by $u$, that is, for all $X\in T_{x}(G/K)$ we have
$$
dE_{u}J_{x}X=J_{ux}dE_{u}X.
$$
\end{definition}

The following result allows us to describe invariant almost complex structures on flag manifolds in terms of complex structure in a simple vector space, namely the tangent space at the origin (trivial coset) of the homogeneous space.

\begin{proposition} \label{prop-corresp-iacs}
There exist an $1-1$ correspondence between a $G$-invariant almost complex structure $J$ and a linear endomorphism $J_{x_{0}}: T_{x_{0}}\mathbb{F} \to T_{x_{0}}\mathbb{F}$ satisfying $J_{x_{0}}^2=-\id$ and commute with the isotropy representation, that is, 
$$
\Ad^{G/K}(k)J_{x_{0}}=J_{x_{0}}\Ad^{G/K}(k), \ \mbox{for all } \ k\in K. 
$$
\end{proposition}

An interesting consequence of the Proposition \ref{prop-corresp-iacs} is that $J(\got{g}_{\alpha})=\got{g}_{\alpha}$, where $\got{g}_\alpha$ is the root space associated to the root $\alpha \in \Pi$. The eigenvalue of $J$ are $\pm \sqrt{-1}$ and the eigenvectors on $\got{m}^{\mathbb{C}}$ are  $X_{\alpha}$, $\alpha\in\Pi$. Therefore $J(X_{\alpha})= \varepsilon_{\alpha} \sqrt{-1} X_{\alpha}$, with $\varepsilon_{\alpha}=\pm 1$ and  $\varepsilon_{\alpha}=-\varepsilon_{-\alpha}$. 

As a consequence of the discussion above we conclude that an invariant almost complex structure on $\mathbb{F}$ is completely described by a set of signals $$\left\{\varepsilon_{\alpha}=\pm 1, \  \alpha\in \Pi_M, \mbox{ satisfying } \varepsilon_{\alpha}=-\varepsilon_{-\alpha}\right\}.$$

\begin{proposition}[\cite{opa},13.4]\label{numest}
Consider the almost complex homogeneous space $M=G/K$  and assume that the isotropy representation admits a decomposition into irreducible and pairwise non-equivalent components, namely,  $\got{m}=\got{m}_{1}\oplus\got{m}_{2}\oplus\cdots\oplus\got{m}_{s}$.  Then $M$ admits $2^{s}$ invariant almost complex structures.  

If we identify the conjugated invariant almost complex structures, then $M$ admits $2^{s-1}$ invariant almost complex structures, up to conjugation.
\end{proposition}

\section{Scalar Curvatures of invariant almost Hermitian structures }\label{scal-hermit}

\subsection{Review: general results about curvatures of almost Hermitian structures} 
Let $(M,J,g)$ be an almost Hermitian manifold with real dimension $2n$, with $J$ being an almost complex structure orthogonal with respect to the Riemannian metric $g$. A linear connection  $\nabla$ on  $M$ is Hermitian if it preserves the metric $g$ and the almost complex structure $J$, that is, $\nabla g=0$ and $\nabla J=0$ (we are not assuming that $J$ is integrable).

Lets us recall the $1$-parameter family of Hermitian connection defined by  Gauduchon in \cite{gau} as follow:
\begin{eqnarray*}
g(\nabla^{t}_{X}Z,Y)&=&g(D_{X}Z,Y)-\frac{1}{2}g(J(D_{X}J)Z,Y)\\
&+&\frac{t}{4}g((D_{JY}J+JD_{Y}J)X,Z)-\frac{t}{4}g((D_{JZ}J+JD_{Z}J)X,Y),
\end{eqnarray*}
where $D$ is the Levi-Civita connection.

There are three special cases: 

\begin{itemize}

\item[i)]$t=0$,  $\nabla^{0}$ is the first canonical Hermitian connection, also known as Lichnerowicz connection or minimal connection. 

\item[ii)] $t=1$, $\nabla^{1}$ is the second canonical Hermitian connection, also known as Chern connection (this connection was used by Chern in the integrable case.

\item[iii)] $t=-1$, $\nabla^{-1}$ is the Bismut connection. In the integrable case, $\nabla^{-1}$ is characterized by its anti-symmetric torsion.

\end{itemize}

According to \cite{fu}, let us define the Hermitian scalar curvatures $s_{1}(t)$ and $s_{2}(t)$ by 
\begin{equation}
s_{1}(t)=R^{t}(u_{\overline{i}},u_{i},u_{j},u_{\overline{j}})
\end{equation}
\begin{equation}
s_{2}(t)=R^{t}(u_{\overline{i}},u_{j},u_{i},u_{\overline{j}})
\end{equation}
where  $R^{t}$ is the curvature tensor associated to $\nabla^{t}$ and $\{u_{i}\}_{i=1,2,\cdots,n}$ is a unitary frame. We call $s_{1}(t)$ and $s_{2}(t)$ the {\em first} and {\em second} Hermitian scalar curvatures, respectively.

For an Hermitian manifold equipped with the Chern connection $\nabla^{1}$ or the Bismut connection $\nabla^{-1}$, the relations between the Hermitian scalar curvatures and Riemannian scalar curvatures were widely studied, see for instance \cite{gau2} e \cite{liu}.

The Nijenhuis tensor $N$ is given by 
$$
N(X,Y)=-[JX,JY]+J[JX,Y]+J[X,JY]+[X,Y], 
$$
where $X,Y\in \Gamma(TM)$.

Let us consider the fundamental (or K\"ahler) 2-form $$F(X,Y)=g(JX,Y).$$

The Lee form $\alpha_{F}$ of $(M,J,g)$ is defined by $$\alpha_{F}=J\delta F,$$ where $\delta=-*d*$ is the codifferential with respect to $g$. The Lee form is also defined by 
$$
dF=(dF)_{0}+\dfrac{1}{n-1}\alpha_{F}\wedge F,
$$
where $(dF)_{0}$ is the primitive part of $dF$.

The covariant derivative of $F$ with respect to the Levi-Civita connection $D$ is
$$
(DF)(X,Y,Z)=\dfrac{1}{2}\left[dF(X,Y,Z)-dF(X,JY,JZ)-N(JX,Y,Z)\right].
$$
Moreover,
$$
(DF)(X,Y,Z)=-(DF)(X,Z,Y)=-(DF)(X,JY,JZ).
$$

The following expression of $DF$ will be very useful for our purposes (see \cite{gau}, \cite{fu}).
\begin{equation*}
(DF)(X,Y,Z)=(dF)^{-}(X,Y,Z)-\dfrac{1}{2}N(JX,Y,Z)+\dfrac{1}{2}\left[(dF)^{+}(X,Y,Z)-(dF)^{+}(X,JY,JZ)\right],
\end{equation*}

where 
$(dF)^{+}$ is the $(1,2)+(2,1)$-part of $dF$ and $(dF)^{-}$ is the $(0,3)+(3,0)$-part of $dF$.

Consider $N^{0}=N-\mathfrak{b}N$, where $\mathfrak{b}$ is the Bianchi projector, $\mathfrak{b}N^{0}=0$ and $\mathfrak{b}N$ is the anti-symmetric part of $N$ defined by  
$$
\mathfrak{b}N(X,Y,Z)=\dfrac{1}{3}\left[N(X,Y,Z)+N(Y,Z,X)+N(Z,X,Y)\right].
$$
According to \cite{gau}, we have $$
3\mathfrak{b}N(X,Y,Z)=(d^{c}F)^{-}(X,Y,Z)=(dF)^{-}(JX,JY,JZ).
$$

The four components described above $(dF)^{-}, N^{0}, (dF)^{+}_{0}$ and $\alpha_{F}$ provide us several geometric information about an almost Hermitian manifold. The $16$ classes of almost Hermitian structures described by Gray-Hervella in \cite{sub} correspond to the vanishing of the some subset  of $$
\{(dF)^{-}, N^{0}, (dF)^{+}_{0},\alpha_{F}\}.
$$

Let us describe some remarkable classes of almost Hermitian structures: 
\begin{itemize}

\item $\{0\}=$ {\em K\"ahler manifolds}: all components vanish, $(dF)^{-}= N^{0}=(dF)^{+}_{0}=\alpha_{F}=0$.

\item $\mathcal{W}_{1}=$ {\em nearly-K\"ahler manifolds}: $N^{0} = (dF)^{+}_{0}= \alpha_{F}=0$.

\item $\mathcal{W}_{1}\oplus \mathcal{W}_{2}=$ {\em $(1,2)$-symplectic manifolds} (or {\it quas}i-K\"ahler): $(dF)^{+}_{0}=\alpha_{F}=0$. 

\item $\mathcal{W}_{1}\oplus \mathcal{W}_{2}\oplus \mathcal{W}_{3}=$ {\em cosymplectic manifolds}: $\alpha_{F}=0$.

\end{itemize}

The Hermitian metric $g$ induces a natural inner product on $\wedge^{k}M$, the bundle of  real $k$-forms, and also on $TM\otimes \wedge^{k}M$ the bundle of $TM$-valuated $k$-forms. The norm of the covariant derivative of the K\"ahler form is given by:
\begin{eqnarray}
\|DF\|^{2}&=&\|dF\|^{2}+\dfrac{1}{4}\|N^{0}\|^{2}-\dfrac{2}{3}\|(dF)^{-}\|^{2}\\
&=&\|(dF)^{+}\|^{2}+\dfrac{1}{4}\|N^{0}\|^{2}+\dfrac{1}{3}\|(dF)^{-}\|^{2}.
\end{eqnarray}
In particular, if $J$ is integrable, then $\|DF\|^{2}=\|dF\|^{2}$.

\begin{theorem}[Theorem 4.3, \cite{fu}]\label{teo1} 
Let $(M,g,J)$ be an almost Hermitian manifold of real dimension $2n$. Then 
\begin{eqnarray*}
s_{1}(t)&=&\dfrac{s}{2}-\dfrac{5}{12}\|(dF)^{-}\|^{2}+\dfrac{1}{16}\|N^{0}\|^{2}+\dfrac{1}{4}\|(dF)_{0}^{+}\|^{2}\\
&+&\left[\dfrac{1}{4(n-1)}+\dfrac{t-1}{2}\right]\|\alpha_{F}\|^{2}+\dfrac{t-2}{2}\delta\alpha_{F}\\
&&\\
s_{2}(t)&=&\frac{s}{2}-\frac{1}{12}\|(dF)^{-}\|^{2}+\frac{1}{32}\|N^{0}\|^{2}-\dfrac{t^{2}-2t}{4}\|(dF)_{0}^{+}\|^{2}\\
&-&\left[\dfrac{t^{2}-2t}{4(n-1)}+\dfrac{(t+1)^{2}}{8}\right]\|\alpha_{F}\|^{2}-\dfrac{t+1}{2}\delta\alpha_{F},
\end{eqnarray*}
where $s$ denotes the Riemannian scalar curvature of $(M,g)$.
\end{theorem}

Note that
\begin{eqnarray}
2s_{1}(t)-s&=&-\dfrac{5}{6}\|(dF)^{-}\|^{2}+\dfrac{1}{8}\|N^{0}\|^{2}+\dfrac{1}{2}\|(dF)_{0}^{+}\|^{2} \label{formula-curvatura}\\
&+&\left[\dfrac{1}{2(n-1)}+(t-1)\right]\|\alpha_{F}\|^{2}+(t-2)\delta\alpha_{F} \nonumber
\end{eqnarray}

\subsection{Curvature of invariant almost Hermitian structures on flag manifolds}

Let us consider a flag manifold $G/K$. We will consider on $G/K$ an invariant metric $g$ and an invariant almost complex structure $J$.  Recall the notation introduced in the Section \ref{sec:prelim}: we will represent an invariant metric $g$ by an $n$-tuple $(\lambda_1,\ldots, \lambda_n)$, where $n$ is the number of irreducible components of the isotropy representation of $\mathfrak{m}=T_o(G/K)$, that is, $\mathfrak{m}=\mathfrak{m}_1\oplus \ldots \oplus \mathfrak{m}_n$. It is worth to point out that if $\{ X_\alpha \}$ is a basis of $\mathfrak{m}$ induced by the Weyl basis then $X_\alpha$ is an eigenvector of the operator $\Lambda$ associated to the metric $g$, with same eigenvalue. This means that every vector in the irreducible component $\mathfrak{m}_i$ has length  $\lambda_i$. The invariant almost complex structure $J$ is parameterized by a set of sign $\{ \varepsilon_\alpha \} =\pm 1 $, with $\varepsilon_{-\alpha}=-\varepsilon_\alpha$, where $\alpha$ is in the set of roots $\Pi$. Each vector $X_\alpha$ is an eigenvector of $J$ with eigenvalue $\varepsilon_\alpha\sqrt{-1}$ and every vector in an irreducible component $\mathfrak{m}_i$ is associated to the same $\varepsilon_i$. 

We will write the Nijenhuis tensor in a similar way as \cite{sn}. We have $N=0$, except in the following situation:
\begin{equation}
g(N(X_{\alpha},X_{\beta}),X_{\gamma}))=-\lambda_{\gamma}m_{\alpha,\beta}(\epsilon_{\alpha}\epsilon_{\beta}+\epsilon_{\alpha}\epsilon_{\gamma}+\epsilon_{\beta}\epsilon_{\gamma}+1),
\end{equation}
where $\alpha+\beta+\gamma=0$.

We also have
\begin{equation}
g((N(X_{\alpha},X_{\beta}),J X_{\gamma})_=-\sqrt{-1}\lambda_{\gamma}m_{\alpha,\beta}(\epsilon_{\alpha}\epsilon_{\beta}\epsilon_{\gamma}+\epsilon_{\alpha}+\epsilon_{\beta}+\epsilon_{\gamma}),
\end{equation}
where $\alpha+\beta+\gamma=0$.

The next result was initially proved by San Martin-Negreiros \cite{sn} in the case of full flag manifolds and by R. de Jesus in her PhD. Thesis for the case of generalized flag manifold (see also \cite{rita}). We include the proof here for the convenience of the reader.   
\begin{lema}\label{lema-sm-neg}
Every invariant almost Hermitian structure $(g,J)$ on a  generalized flag manifolds is cosymplectic. 
\end{lema}
\begin{proof}
An almost Hermitian structure is cosymplectic if and only if
$$
\alpha_{F}(X)=\dfrac{1}{2n-1}\sum_{i}dF(X,X_{i},Y_{i})=0,
$$
where $\{X_{i}\}$ is a basis of the tangent space, $\{Y_{i}\}$ is a basis of the dual space with relation a nondegenerate form $F$. We take $\{X_{i}\}=\{A_{\alpha}, \sqrt{-1}S_{\alpha};\alpha\in \Pi^{+}_M \}$ and $\{Y_{i}\}=\{\sqrt{-1}S_{\alpha}, A_{\alpha};\alpha\in \Pi^{+}_M \} $. So,
\begin{eqnarray*}
dF(X,X_{\alpha}-X_{-\alpha},iX_{\alpha}+iX_{-\alpha})&=&idF(X,X_{\alpha},X_{\alpha})+idF(X,X_{\alpha},X_{-\alpha})\\
&-&idF(X,X_{-\alpha},X_{\alpha})-idF(X,X_{-\alpha},X_{-\alpha})\\
&=&2i(dF(X,X_{\alpha},X_{\alpha})+dF(X,X_{\alpha},X_{-\alpha})).
\end{eqnarray*}
We know that $dF(X_{\alpha},X_{\beta}, X_{\gamma})=0$ unless $\alpha+\beta+\gamma=0$.
Thus, we take $X=X_{\beta}$ and
$$
\left\{
\begin{array}{l}
dF(X_{\beta},X_{\alpha}, X_{\alpha})=0,\ \mbox{ because if} \ \alpha+\beta+\alpha=0 \ \mbox{then} \ \beta=-2\alpha.\\  
dF(X_{\beta},X_{\alpha}, X_{-\alpha})=0,\ \mbox{because if} \ \alpha+\beta-\alpha=0 \ \mbox{then} \ \beta=0
\end{array}
\right\},
$$
and it is a contradiction. 
Therefore, for every root $\gamma$,
$$
\alpha_{F}(X)=\dfrac{1}{2n-1}\sum_{\alpha>0}dF(X_{\gamma},X_{\alpha}-X_{-\alpha},iX_{\alpha}+iX_{-\alpha})=0.
$$
\end{proof}
\begin{obs}
1. According to Lemma \ref{lema-sm-neg}, the Lee form vanishes identically for every generalized flag manifold. Therefore we have $dF=(dF)_{0}$ and hence $(dF)^{+}=(dF)_{0}^{+}$. From now on, we will use $(dF)^{+}$ instead of $(dF)_{0}^{+}$.

2. We have $\delta \alpha_{F}=0$, where $\delta$  is the codifferential. In this case $s_{1}(t)$ is independent of $t$ (see Theorem \ref{teo1}), and we will denote $s_{1}(t)$ simply by $s_{1}$. In this case, $s_1$ coincides with the Chern scalar curvature $s_C$ 

3. In Theorem \ref{teo1} we will omit the terms which $\alpha_{F}$ appears. 
\end{obs}

The next definition was introduced in \cite{sn} in order to obtain the classification of invariant almost Hermitian structures on full flag manifolds, and we will use this concept widely.

\begin{definition}
Let $J$ be an invariant almost complex structure on the flag manifold $G/K$. The triple of roots $\alpha, \beta, \gamma\in \Pi^+_M$, with  $\alpha+\beta+\gamma=0$ is said to be a $(0,3)$-triple if $\varepsilon_\alpha=\varepsilon_\beta=\varepsilon_\gamma$. It is a $(1,2)$-triple otherwise.
\end{definition}

In order to calculate the Hermitian scalar curvatures, the first step is computing the norms of $N^{0}$, $(dF)^{-}$, $(dF)^{+}$ e $\alpha_{F}$, as described in \cite{fu}. Since these computations are standard we omit the details.

\begin{proposition}[\cite{nosso2}] \label{prop-norma}
Let $(g,J)$ be an invariant almost Hermitian structure in the generalized flag manifold $G/K$. We have
\begin{eqnarray*}
\|N^{0}\|^{2}&=&\sum_{\alpha+\beta+\gamma=0}\left(m_{\alpha,\beta}\right)^{2}\dfrac{(\epsilon_{\alpha}\epsilon_{\beta}\epsilon_{\gamma}+\epsilon_{\alpha}+\epsilon_{\beta}+\epsilon_{\gamma})^2}{54\lambda_{\alpha}\lambda_{\beta}\lambda_{\gamma}}[(-2\lambda_{\alpha}+\lambda_{\beta}+\lambda_{\gamma})^{2}\\
& &+ (-2\lambda_{\gamma}+\lambda_{\alpha}+\lambda_{\beta})^{2}+(-2\lambda_{\beta}+\lambda_{\alpha}+\lambda_{\gamma})^{2}], \\ \\
\|(dF)^{-}\|^{2}&=&\sum_{\alpha+\beta+\gamma=0}\dfrac{\left(m_{\alpha,\beta}\right)^{2} (\epsilon_{\alpha}\epsilon_{\beta}\epsilon_{\gamma}+\epsilon_{\alpha}+\epsilon_{\beta}+\epsilon_{\gamma}  )^2\left(\lambda_{\alpha}+\lambda_{\beta}+\lambda_{\gamma}\right)^{2}}{96\lambda_{\alpha}\lambda_{\beta}\lambda_{\gamma}}, \\ \\
\|DF\|^{2}&=&\sum_{\alpha+\beta+\gamma=0}\dfrac{m_{\alpha,\beta}^{2}}{3\lambda_{\alpha}\lambda_{\beta}\lambda_{\gamma}}\left[\dfrac{1}{4}(\varepsilon_{\beta}+\varepsilon_{\gamma})^{2}(-\lambda_{\alpha}+\lambda_{\beta}+\lambda_{\gamma})^{2}\right.\\
& & + \dfrac{1}{4}(\varepsilon_{\alpha}+\varepsilon_{\gamma})^{2}(\lambda_{\alpha}-\lambda_{\beta}+\lambda_{\gamma})^{2}
+\left.\dfrac{1}{4}(\varepsilon_{\alpha}+\varepsilon_{\beta})^{2}(\lambda_{\alpha}+\lambda_{\beta}-\lambda_{\gamma})^{2}\right], \\ \\
\|(dF)^{+}\|^{2}&=&\sum_{\alpha+\beta+\gamma=0}\dfrac{m_{\alpha,\beta}^{2}\{4(\varepsilon_{\alpha}\lambda_{\alpha}+\varepsilon_{\beta}\lambda_{\beta}+\varepsilon_{\gamma}\lambda_{\gamma})-(\varepsilon_{\alpha}+\varepsilon_{\beta}+\varepsilon_{\gamma}+\varepsilon_{\alpha}\varepsilon_{\beta}\varepsilon_{\gamma})(\lambda_{\alpha}+\lambda_{\beta}+\lambda_{\gamma})\}}{96\lambda_{\alpha}\lambda_{\beta}\lambda_{\gamma}}^{2}.
\end{eqnarray*}

\end{proposition}

\section{Flag manifolds with two isotropy summands}\label{sec:flag2somandos}

In this section we will discuss the equation $2s_{1}-s=0$ for generalized flag manifolds $G/K$ whose the isotropy representation splits into two irreducible and non-equivalent components, namely $\mathfrak{m}=\mathfrak{m}_{1}\oplus\mathfrak{m}_{2}$.
The classification of this family of homogeneous space is given as follows: recall the {\em height} of a simple root $\alpha_p$ is the integer that appear as its coefficient in the highest root, that is, if $\mu=c_1{\alpha_1}+\ldots + c_p\alpha_p+\ldots +c_n\alpha_n$ is the highest root then $\htt (\alpha_p)=c_p$. A generalized flag manifold has two istropy components iff it is parameterized by the set $\Theta\subset\Sigma$ such that $\Theta=\Sigma \setminus\{ \alpha_p\}$ with $\htt(\alpha_p)=2$ (see Section \ref{sec:prelim}).  This class of homogeneous spaces is listed in Table \ref{tab:two-summands} (see \cite{anas}).  

\begin{table}[h]
\caption{Flag manifolds with two isotropy summnads}
    \label{tab:two-summands}
    \begin{tabular}{l}
       $SO(2\ell+1)/U(\ell-m)\times SO(2m+1),\quad \ell>0, m\geq 0, \ell-m\neq 1$    \\ \hline
       $Sp(\ell)/U(\ell-m)\times Sp(m), \quad \ell >0, m>0$ \\ \hline
       $SO(2\ell)/U(\ell-m)\times SO(2m), \quad \ell >0, m>0, \ell-m\neq 1$ \\ \hline
       $G_2/U(2)$, $U(2)$ represented by the short root of $G_2$ \\ \hline
       $F_4/SO(7) \times U(1)$ \\ \hline
       $F_4/Sp(3) \times U(1)$ \\ \hline
       $E_6/SU(6) \times U(1)$ \\ \hline
       $E_6/SU(2) \times SU(5) \times U(1) $ \\ \hline
       $E_7/SU(7) \times U(1) $ \\ \hline
       $E_7/SU(2) \times SO(10) \times U (1)$ \\ \hline
       $E_7/SO(12) \times U(1)$ \\ \hline
       $E_8/E_7 \times U(1)$ \\ \hline
       $E_8/SO(14) \times U(1)$
    \end{tabular}
    
\end{table}


The isotropy summands $\got{m}_{1}$ and $\got{m}_{2}$ satisfy the following relations: 
$$
[\mathfrak{m}_{1},\mathfrak{m}_{1}]\subset \mathfrak{k}\oplus\mathfrak{m}_{2}, \ \ \ \ [\mathfrak{m}_{1},\mathfrak{m}_{2}]\subset \mathfrak{m}_{1} \ \ \ \ \mbox{and} \ \ \ \  [\mathfrak{m}_{2},\mathfrak{m}_{2}]\subset \mathfrak{k}.
$$

Note that the zero sum  triples can be parameterized by:
$$
\begin{array}{ll}
\alpha+\beta+\gamma=0 & \mbox{where}\ \alpha\in \mathfrak{m}_{1}, \ \beta\in\mathfrak{m}_{1} \ \mbox{and} \ \gamma\in\mathfrak{m}_{2}.
\end{array}
$$

For flag manifolds $G/K$ with two isotropy summands we have two invariant almost complex structures (up to conjugation), namely $J=(+,+)$ and $J_1=(+,-)$. Let us consider our analysis for $J$ and $J_1$ separately.

Consider complex (integrable) structure $J=(+,+)$. The (1,2)-triples are given by $\mathfrak{m}_{1}+\mathfrak{m}_{1}-\mathfrak{m}_{2}=0$ and there is no (0,3)-triples. We define $$L=\sum_{\substack{\{\delta,\sigma,\eta\} \ \rm{is} \\ (1,2)-\rm{triple}}}m_{\delta,\sigma}^{2},$$
where the real numbers $m_{\delta,\sigma}$ are the structural constants induced by the Weyl basis.

Let us consider an invariant metric $g$ parameterized by $g=(\lambda_1,\lambda_2)$. We get
$$\|(dF)^{-}\|^{2}=\|N^{0}\|^{2}=0, \quad \|(dF)^{+}\|^{2}=L\dfrac{(2\lambda_1-\lambda_2)^{2}}{6\lambda_1^{2}\lambda_2}.$$

If the metric $g$ satisfies $\lambda_2=2\lambda_1$ then $(J,g)$ is K\"ahler, otherwise $(J,g)\in\mathcal{W}_3$. The expression $2s_{1}-s$ reads:
\begin{eqnarray*}
2s_{1}-s&=&-\dfrac{5}{6}\|(dF)^{-}\|^{2}+\dfrac{1}{8}\|N^{0}\|^{2}+\dfrac{1}{2}\|(dF)^{+}\|^{2}\\
&=&L\dfrac{(2\lambda_1-\lambda_2)^{2}}{12\lambda_1^{2}\lambda_2}.
\end{eqnarray*}

Thus, $2s_{1}-s=0$ if, and only if, $\lambda_2=2\lambda_1$ and $(J,g)$ is K\"ahler. 
 
 \

Let us consider now the invariant almost complex structure $J_1=(+,-)$ (non-integrable). The (0,3)-triples are given by $\mathfrak{m}_{1}+\mathfrak{m}_{1}-\mathfrak{m}_{2}=0$, and there is no (1,2)-triples. We define
$$K=\sum_{\substack{\{\alpha,\beta,\gamma\}\ \rm{is} \\  (0,3)-\rm{triple}}}m_{\alpha,\beta}^{2},$$
where the real numbers $m_{\alpha,\beta}$ are the structural constants induced by the Weyl basis.

We have
$$\|(dF)^{+}\|^{2}=0, \quad\|(dF)^{-}\|^{2}=K\dfrac{(2\lambda_1+\lambda_2)^{2}}{6\lambda_1^{2}\lambda_2}, \quad \|N^{0}\|^{2}=8K\dfrac{2(-\lambda_1+\lambda_2)^{2}+(2\lambda_1-2\lambda_2)^{2}}{27\lambda_1^{2}\lambda_2}.$$

If the metric $g$ satisfies $\lambda_2=\lambda_1$ then $(J_1,g) \in \mathcal{W}_1$, otherwise $(J_1,g)\in\mathcal{W}_1\oplus \mathcal{W}_2$. The expression $2s_{1}-s$ reads:
\begin{eqnarray*}
2s_{1}-s&=&-\dfrac{5}{6}\|(dF)^{-}\|^{2}+\dfrac{1}{8}\|N^{0}\|^{2}+\dfrac{1}{2}\|(dF)^{+}\|^{2}\\
&=&-5K\dfrac{(2\lambda_1+\lambda_2)^2}{36\lambda_1^{2}\lambda_2}
+K\dfrac{(-2\lambda_1+\lambda_1+\lambda_2)^{2}+(\lambda_1-2\lambda_1+\lambda_2)^{2}+(\lambda_1+\lambda_1-2\lambda_2)^{2}}{27\lambda_1^{2}\lambda_2}\\
&=&\frac{K \left(-4 \lambda_1^2-12 \lambda_1 \lambda_2+\lambda_2^2\right)}{12 \lambda_1^2 \lambda_2}.
\end{eqnarray*}

Therefore, if $(J_1, g)\in\mathcal{W}_1$ there is not solution for equation $2s_1-s=0$. If $(J_1, g)\in\mathcal{W}_1\oplus \mathcal{W}_2$ then $\lambda_2=2\lambda_1( \sqrt{10}+3)$ is a solution of $2s_1-s=0$. 

Summing up we obtain the following result:
\begin{theorem}\label{teo-2somandos}
Let $G/K$ be a generalized flag manifold with two isotropy summands, and let $g=(\lambda_1,\lambda_2)$ be an invariant Riemannian metric. Then $2s_1-s=0$ holds for $G/K$ if the pair $(g,J_i)$ satisfies the following conditions: 

\begin{itemize}
    \item For the invariant complex struture $J=(+,+)$, the metric $g$ is parameterized by $\lambda_2=2\lambda_1$. In this case $(g,J)$ is K\"ahler.
    \item For the invariant almost complex structure $J_1=(+,-)$, the metric $g$ is parameterized by $\lambda_2=2\lambda_1( \sqrt{10}+3)$. In this case, $(g, J_1)\in\mathcal{W}_1\oplus \mathcal{W}_2$. 
\end{itemize}
\end{theorem}

\bigskip

\section{Flag manifolds with three isotropy summands}\label{sec:flag3somandos}

Generalized flag manifolds whose the isotropy representation decompose into three pairwise
non-isomorphic irreducible components, i.e., $\mathfrak{m}=\mathfrak{m}_{1}\oplus\mathfrak{m}_{2}\oplus\mathfrak{m}_{3}$ are classified by Kimura in \cite{kimura} and described in terms of the choice of $\Theta\subset\Sigma$ in Table \ref{table:3compo-sigma}. It is worth to point out there are two sub-families of flags with three isotropy components, namely Type I and II, see Tables  \ref{table:3comp-tipo1} and \ref{table:3comp-tipo2}.

\begin{table}[!htb] 

\centering

\caption{Types of flag manifolds with three summands}\label{table:3compo-sigma} 

\begin{tabular}{|c|c|} 

\hline 

Type & $\Theta \subset \Sigma$\\ 
\hline
\hline
$I$ & $\Theta=\Sigma\setminus\{\alpha_{p}:\htt(\alpha_{p})=3\}$\\
\hline
$II$ & $\Theta=\Sigma\setminus \{\alpha_{p},\alpha_{q} : \htt(\alpha_{p})=\htt(\alpha_{q})=1\}$\\
\hline
\end{tabular}
\label{tab2}
\end{table}

\begin{table}[!htb] \label{3som-tipo1}

\centering

\caption{The dimensions $d_{i}=\dim(\mathfrak{m}_{i})$ for any $M=G/K$ of type I} \label{table:3comp-tipo1} 

\begin{tabular}{|c|c|c|c|} 

\hline 

 Flag manifold $G/K$ & $d_{1}$ & $d_{2}$ &  $d_{3}$\\ 

\hline
\hline
$G_{2}/U(2)$ & $4$ & $2$ &  $4$\\
\hline
$F_{4}/SU(3)\times SU(2)\times U(1)$   & $24$ & $12$ &  $4$\\
\hline
$E_{6}/SU(3)\times SU(3)\times SU(2)\times U(1)$  & $36$ & $18$ &  $4$\\
\hline
 $E_{7}/SU(5)\times SU(3)\times U(1)$  & $60$ & $30$ &  $8$\\
 $E_{7}/SU(6)\times SU(2)\times U(1)$  & $60$ & $30$ &  $4$\\
\hline
$E_{8}/E_{6}\times SU(2)\times U(1)$  & $108$ & $54$ &  $4$\\
$E_{8}/SU(8)\times U(1)$  & $112$ & $56$ &  $16$\\
\hline
\end{tabular}
\end{table}

\begin{table}[!htb] 

\centering

\caption{Flag manifolds with three summands of type II}\label{table:3comp-tipo2} 

\begin{tabular}{|c|c|c|c|} 

\hline 

Flag manifold $G/K$& $d_{1}$ & $d_{2}$ & $d_{3}$\\ 
\hline
\hline
$SU(l+m+n)/S(U(l)\times U(m)\times U(n) \ (l,m,n\in \mathbb{Z}^{+})$ & $2mn$ & $2mp$ & $2np$\\
\hline
$SO(2l)/U(1)\times U(l-1) \ (l \geq 4)$ & $2(l-1)$ & $2(l-1)$ & $(l-1)(l-2)$\\
\hline
$E_{6}/SO(8)\times U(1)\times U(1)$& 16 & 16 & 16\\
\hline
\end{tabular}
\end{table}

 Flag manifolds with three isotropic summands of type I the summands satisfy the following relations with respect to $\got{m}_{1}$, $\got{m}_{2}$ and $\got{m}_{3}$: 
\begin{eqnarray*}
&[\mathfrak{m}_{1},\mathfrak{m}_{1}]\subset \mathfrak{k}\oplus\mathfrak{m}_{2}, \ \ \ [\mathfrak{m}_{1},\mathfrak{m}_{2}]\subset \mathfrak{m}_{1}\oplus \mathfrak{m}_{3} \ \ \ [\mathfrak{m}_{1},\mathfrak{m}_{3}]\subset \mathfrak{m}_{2} \\ 
&[\mathfrak{m}_{2},\mathfrak{m}_{2}]\subset \mathfrak{k}, \ \ \ [\mathfrak{m}_{2},\mathfrak{m}_{3}]\subset \mathfrak{m}_{1} \ \ \ \mbox{and} \ \ \ [\mathfrak{m}_{3},\mathfrak{m}_{3}]\subset \mathfrak{k}.
\end{eqnarray*}

In the case of flag manifolds with three isotropic summands of type II the components $\got{m}_{1}$, $\got{m}_{2}$ and $\got{m}_{3}$ satisfy the following inclusions: 
\begin{eqnarray*}
&[\mathfrak{m}_{1},\mathfrak{m}_{1}]\subset \mathfrak{k}, \ \ \ [\mathfrak{m}_{1},\mathfrak{m}_{2}]= \mathfrak{m}_{3} \ \ \ [\mathfrak{m}_{1},\mathfrak{m}_{3}]=\mathfrak{m}_{2} \\ 
&[\mathfrak{m}_{2},\mathfrak{m}_{2}]\subset \mathfrak{k}, \ \ \ [\mathfrak{m}_{2},\mathfrak{m}_{3}]= \mathfrak{m}_{1} \ \ \ \mbox{and} \ \ \ [\mathfrak{m}_{3},\mathfrak{m}_{3}]\subset \mathfrak{k}.
\end{eqnarray*}


Consequently, there are two groups of triples with zero sum:
$$
\begin{array}{ll}
\alpha+\beta+\gamma=0 & \mbox{where}\ \alpha\in \mathfrak{m}_{1}, \ \beta\in\mathfrak{m}_{1} \ \mbox{and} \ \gamma\in\mathfrak{m}_{2};\\
\alpha+\beta+\gamma=0 & \mbox{where}\ \alpha\in \mathfrak{m}_{1}, \ \beta\in\mathfrak{m}_{2} \ \mbox{and} \ \gamma\in\mathfrak{m}_{3}.
\end{array}
$$

\begin{definition}
Let $G/K$ be a generalized flag manifold equipped with an invariant almost complex structure $J$ and Riemannian metric $g$. We define
\begin{equation}\label{constanteKL}
K=\sum_{\substack{\{\alpha,\beta,\gamma\}\ \rm{is} \\  (0,3)-\rm{triple}}}m_{\alpha,\beta}^{2}, \hspace{0.5cm} {\rm and} \hspace{0.5cm}
L=\sum_{\substack{\{\delta,\sigma,\eta\} \ \rm{is} \\ (1,2)-\rm{triple}}}m_{\delta,\sigma}^{2},\\
\end{equation}
where the real numbers $m_{\alpha,\beta}$, $m_{\delta,\sigma}$ are the structure constants induced by the Weyl basis.
\end{definition}

Recall we have four invariant almost complex structure (up to conjugation)\\ $J_1 =(+,+,+)$, $J_2=(-,+,+)$, $J_3=(+,-,+)$, $J_4=(+,+,-)$ on a flag manifold $G/K$ with three isotropic summands (cf. \cite{kimura}). We will analyze the solution of the equation $2s_1-s=0$ separately for each invariant almost complex structure $J_i$, $i=1,\ldots, 4$.   

\begin{table}[h]
\caption{Invariant almost complex structure $J_1=(+,+,+)$}
\label{table-J1}
\begin{tabular}{|c|c|c|cc|}
\hline
\multirow{2}{*}{\begin{tabular}[c]{@{}c@{}}Flag\\ Manifold\end{tabular}} & \multirow{2}{*}{\begin{tabular}[c]{@{}c@{}}$(1,2)$-triples\\ $(\alpha+\beta+\gamma)=0$\end{tabular}}                                      & \multirow{2}{*}{\begin{tabular}[c]{@{}c@{}}$(0,3)$-triples\\ $(\alpha+\beta+\gamma)=0$\end{tabular}} & \multicolumn{2}{c|}{\begin{tabular}[c]{@{}c@{}}Gray-Hervella\\ Classification\end{tabular}}                                                         \\ \cline{4-5} 
                                                                         &                                                                                                                                           &                                                                                                      & \multicolumn{1}{c|}{Class}                                                  & Condition on the metric                                                            \\ \hline
Type I                                                                   & \begin{tabular}[c]{@{}c@{}}$\alpha \in m_1, \beta\in \mathfrak{m}_1, \gamma\in \mathfrak{m}_2$\\  or\\ $\alpha \in \mathfrak{m}_1, \beta\in \mathfrak{m}_2, \gamma\in \mathfrak{m}_3$\end{tabular} & \ding{55}                                                                                                 & \multicolumn{1}{c|}{\begin{tabular}[c]{@{}c@{}}K\"ahler\\ $\mathcal{W}_3$\end{tabular}} & \begin{tabular}[c]{@{}c@{}}$\lambda_2=2\lambda_1$ and $\lambda_3=3\lambda_1$\\ Otherwise\end{tabular} \\ \hline
Type II                                                                  & $\alpha \in \mathfrak{m}_1, \beta\in \mathfrak{m}_2, \gamma\in \mathfrak{m}_3$                                                                                              & \ding{55}                                                                                                 & \multicolumn{1}{c|}{\begin{tabular}[c]{@{}c@{}}K\"ahler\\ $\mathcal{W}_3$\end{tabular}} & \begin{tabular}[c]{@{}c@{}}$\lambda_3=\lambda_1+\lambda_2$\\ Otherwise\end{tabular}           \\ \hline
\end{tabular}
\end{table}

\begin{theorem}\label{teo3somandosJ1}
Let $G/K$ be a generalized flag manifold with three isotropy summands equipped with the invariant almost complex structure $J_1$ described in Table \ref{table-J1} and invariant Riemannian metric $g=(\lambda_1, \lambda_2,\lambda_3)$. Then $2s_1-s=0$ holds for $(G/K, g, J_1)$ if the metric $g$ satisfies the following conditions: 
\begin{enumerate}
    \item $G/K$ is type I
    \begin{enumerate}
        \item $g=(\lambda_1, 2\lambda_1, 3\lambda_1)$. In this case, $(g,J_1)$ is K\"ahler.
    \end{enumerate}
    \item $G/K$ is type II
    \begin{enumerate}
        \item $g=(\lambda_1,\lambda_2, \lambda_3)$ with $\lambda_3=\lambda_1+\lambda_2$. In this case, $(g,J_1)$ is K\"ahler. 
    \end{enumerate}
\end{enumerate}
\end{theorem}

\begin{proof}
Let us split our analysis for flag manifolds of type I and II:
\begin{enumerate}
\item[(a)] {\bf Flag manifolds of type I: } 

Let us define
$$L_{1}=\sum_{\substack{\{\delta,\sigma,\eta\} \ \rm{is} \\ (1,2)-\rm{triple}}}m_{\delta,\sigma}^{2}, \ \ \mbox{where} \ \ \delta\in \mathfrak{m}_{1}, \ \sigma\in\mathfrak{m}_{1} \ \mbox{and} \ \eta\in\mathfrak{m}_{2}\\
$$

$$L_{2}=\sum_{\substack{\{\delta,\sigma,\eta\} \ \rm{is} \\ (1,2)-\rm{triple}}}m_{\delta,\sigma}^{2}, \ \ \mbox{where} \ \ \delta\in \mathfrak{m}_{1}, \ \sigma\in\mathfrak{m}_{2} \ \mbox{and} \ \eta\in\mathfrak{m}_{3}\\
$$

We obtain
$\|(dF)^{-}\|^{2}=\|N^{0}\|^{2}=0$.

$\|(dF)^{+}\|^{2}= L_1\dfrac{(\lambda_{1}+\lambda_{1}-\lambda_{2})^{2}}{6\lambda_{1}^{2}\lambda_{3}}+L_2\dfrac{(\lambda_{1}+\lambda_{2}-\lambda_{3})^{2}}{6\lambda_{1}\lambda_{2}\lambda_{3}}$.

It follows that, 
\begin{eqnarray*}
2s_{1}-s&=&-\dfrac{5}{6}\|(dF)^{-}\|^{2}+\dfrac{1}{8}\|N^{0}\|^{2}+\dfrac{1}{2}\|(dF)^{+}\|^{2}\\
&=&L_1\dfrac{(\lambda_{1}+\lambda_{1}-\lambda_{2})^{2}}{12\lambda_{1}^{2}\lambda_{2}} + L_2\dfrac{(\lambda_{1}+\lambda_{2}-\lambda_{3})^{2}}{12\lambda_{1}\lambda_{2}\lambda_{3}}\geq 0
\end{eqnarray*}
Therefore, the equation $2s_{1}-s=0$ there is no solution at least the pair $(g,J)$ is Kähler what we already well known. 

\

\item[(b)] {\bf Flag manifolds of type II: } 

We start by computing the components of the covariant derivative of the K\"ahler form. By a direct computation we get
\begin{equation}
    \|(dF)^{-}\|^{2}=\|N^{0}\|^{2}=0, \quad {\rm and} \quad \|(dF)^{+}\|^{2}=L\dfrac{(\lambda_{1}+\lambda_{2}-\lambda_{3})^{2}}{6\lambda_{1}\lambda_{2}\lambda_{3}},
\end{equation}
where the non-zero constant $L$ is defined by equation \ref{constanteKL} and depends on the flag manifold.

By equation \ref{formula-curvatura} we find
\begin{eqnarray*}
2s_{1}-s&=&-\dfrac{5}{6}\|(dF)^{-}\|^{2}+\dfrac{1}{8}\|N^{0}\|^{2}+\dfrac{1}{2}\|(dF)^{+}\|^{2}\\
&=&L\dfrac{(\lambda_{1}+\lambda_{2}-\lambda_{3})^{2}}{12\lambda_{1}\lambda_{2}\lambda_{3}}\geq 0
\end{eqnarray*}

Hence, there is no solution for the equation $2s_{1}-s=0$ unless the pair $(g,J)$ is K\"ahler, and this solution is already well known. 

\end{enumerate}
\end{proof}

\begin{table}[h]
\caption{Invariant almost complex structure $J_2=(-,+,+)$}
\label{table-J2}
\begin{tabular}{|c|c|c|cc|}
\hline
\multirow{2}{*}{\begin{tabular}[c]{@{}c@{}}Flag\\ Manifold\end{tabular}} & \multirow{2}{*}{\begin{tabular}[c]{@{}c@{}}$(1,2)$-triples\\ $(\alpha+\beta+\gamma)=0$\end{tabular}} & \multirow{2}{*}{\begin{tabular}[c]{@{}c@{}}$(0,3)$-triples\\ $(\alpha+\beta+\gamma)=0$\end{tabular}} & \multicolumn{2}{c|}{\begin{tabular}[c]{@{}c@{}}Gray-Hervella\\ Classification\end{tabular}}                                                                                                                 \\ \cline{4-5} 
                                                                         &                                                                                                      &                                                                                                      & \multicolumn{1}{c|}{Class}                                                                                                 & Condition on the metric                                                                     \\ \hline
Type I                                                                   & $\alpha \in \mathfrak{m}_1, \beta\in \mathfrak{m}_1, \gamma\in \mathfrak{m}_2$                                                        & $\alpha \in \mathfrak{m}_1, \beta\in \mathfrak{m}_2, \gamma\in \mathfrak{m}_3$                                                        & \multicolumn{1}{c|}{\begin{tabular}[c]{@{}c@{}}$\mathcal{W}_1\oplus \mathcal{W}_2$\\ $\mathcal{W}_1\oplus \mathcal{W}_3$\\ $\mathcal{W}_1\oplus \mathcal{W}_2\oplus \mathcal{W}_3$\end{tabular}} & \begin{tabular}[c]{@{}c@{}}$\lambda_3=-\lambda_1+\lambda_2$ and $\lambda_2>\lambda_1$\\ $\lambda_1=\lambda_2$\\ Otherwise\end{tabular} \\ \hline
Type II                                                                  & $\alpha \in \mathfrak{m}_1, \beta\in \mathfrak{m}_2, \gamma\in \mathfrak{m}_3$                                                        & \ding{55}                                                                                                 & \multicolumn{1}{c|}{\begin{tabular}[c]{@{}c@{}}K\"ahler\\ $\mathcal{W}_3$\end{tabular}}                                                & \begin{tabular}[c]{@{}c@{}}$\lambda_3=\lambda_1+\lambda_2$\\ Otherwise\end{tabular}                    \\ \hline
\end{tabular}
\end{table}

\begin{theorem}\label{teo3somandosJ2}
Let $G/K$ be a generalized flag manifold with three isotropy summands equipped with the invariant almost complex structure $J_2$ described in Table \ref{table-J2} and invariant Riemannian metric $g=(\lambda_1, \lambda_2,\lambda_3)$. Then $2s_1-s=0$ holds for $(G/K, g, J_2)$ if the metric $g$ satisfies the following conditions: 

\begin{enumerate}
    \item $G/K$ is type I
    \begin{enumerate}
       
        \item $\lambda_{1}<\lambda_{2}<2(\sqrt{10}+3) \lambda_{1}$ and\\
$\lambda_{3}=\frac{-\sqrt{K \left(4 \lambda_{1}^2+12 \lambda_{1} \lambda_{2}-\lambda_{2}^2\right) \left(K \left(4 \lambda_{1}^2+12 \lambda_{1} \lambda_{2}-\lambda_{2}^2\right)+4 L \lambda_1 (\lambda_{2}-\lambda_{1})\right)}+K \left(4 \lambda_{1}^2+12 \lambda_{1} \lambda_{2}-\lambda_{2}^2\right)-2 L \lambda_{1}^2+2 L \lambda_{1} \lambda_{2}}{2 L \lambda_{1}}$. In this case, $(g,J_2)$ belongs to the Gray-Hervella class $\mathcal{W}_{1}\oplus \mathcal{W}_{2}\oplus \mathcal{W}_{3}$.

        \item $g=(\lambda_1, \lambda_2, \lambda_3)$ with $\lambda_2=(2\sqrt{10}+6)\lambda_1$ and $\lambda_3=\lambda_2-\lambda_1$. In this case, $(g,J_2)$ belongs to the Gray-Hervella class $\mathcal{W}_{1}\oplus \mathcal{W}_{2}$.
        \item $g=(\lambda_1, \lambda_2, \lambda_3)$ with $\lambda_1=\lambda_2=\frac{L \lambda_3}{15 K}$.  In this case, $(g,J_2)$ belongs to the Gray-Hervella class $\mathcal{W}_{1}\oplus \mathcal{W}_{3}$.

    \end{enumerate}
    \item $G/K$ is type II
    \begin{enumerate}
        \item $g=(\lambda_1,\lambda_2, \lambda_3)$ with $\lambda_2=\lambda_1+\lambda_3$. In this case, $(g,J_2)$ is K\"ahler. 
    \end{enumerate}
\end{enumerate}
\end{theorem}

\begin{proof}
We will proceed as in the proof of Theorem \ref{teo3somandosJ1}. Let us split our analysis for flag manifolds of type I and II:
\begin{enumerate}

\item[a)] {\bf Flag manifolds of type I:}

We get:
\begin{equation}
  \|(dF)^{-}\|^{2}=K\dfrac{(2\lambda_{1}+\lambda_{2})^{2}}{6\lambda_{1}^{2}\lambda_{2}}, \quad \quad \|(dF)^{+}\|^{2}=L\dfrac{(\lambda_{1}-\lambda_{2}+\lambda_{3})^{2}}{6\lambda_{1}\lambda_{2}\lambda_{3}}  
\end{equation}
and
\begin{equation}
  \|N^{0}\|^{2}=8K\dfrac{(-2\lambda_{1}+\lambda_{1}+\lambda_{2})^{2}+(\lambda_{1}-2\lambda_{1}+\lambda_{2})^{2}+(2\lambda_{1}-2\lambda_{2})^{2}}{27\lambda_{1}^{2}\lambda_{2}}  
\end{equation},
where the non-zero constants $K$ and $L$ are defined by equation \ref{constanteKL} and depends on the flag manifold.

Let us analyze case-by-case the Gray-Hervella class with respect to the almost complex structure $J_2$: 

\

\begin{enumerate}
\item[(I)] For $\mathcal{W}_{1}\oplus \mathcal{W}_{2}\oplus \mathcal{W}_{3}$: 
\begin{eqnarray*}
2s_{1}-s&=&-\dfrac{5}{6}\|(dF)^{-}\|^{2}+\dfrac{1}{8}\|N^{0}\|^{2}+\dfrac{1}{2}\|(dF)^{+}\|^{2}\\
&=&-5K\dfrac{(2\lambda_{1}+\lambda_{2})^{2}}{36\lambda_{1}^{2}\lambda_{2}}+K\dfrac{(-2\lambda_{1}+\lambda_{1}+\lambda_{2})^{2}+(\lambda_{1}-2\lambda_{1}+\lambda_{2})^{2}+(2\lambda_{1}-2\lambda_{2})^{2}}{27\lambda_{1}^{2}\lambda_{2}}\\
&+&L\dfrac{(\lambda_{1}-\lambda_{2}+\lambda_{3})^{2}}{12\lambda_{1}\lambda_{2}\lambda_{3}}
\end{eqnarray*}

Thus, $2s_1-s=0$ if, and only if $$K(-4\lambda_1^2 -12 \lambda_1\lambda_2 +\lambda_2^2)\lambda_3 + L \, \lambda_1(\lambda_1 -\lambda_2 +\lambda_3)^2=0.$$

Examples of solutions: $g=(\lambda_{1},\lambda_{2},\lambda_{3})$ parameterized by
\begin{enumerate}
\item $\lambda_{1}<\lambda_{2}<2(\sqrt{10}+3) \lambda_{1}$ and\\
$\lambda_{3}=\frac{-\sqrt{K \left(4 \lambda_{1}^2+12 \lambda_{1} \lambda_{2}-\lambda_{2}^2\right) \left(K \left(4 \lambda_{1}^2+12 \lambda_{1} \lambda_{2}-\lambda_{2}^2\right)+4 L \text{l1} (\lambda_{2}-\lambda_{1})\right)}+K \left(4 \lambda_{1}^2+12 \lambda_{1} \lambda_{2}-\lambda_{2}^2\right)-2 L \lambda_{1}^2+2 L \lambda_{1} \lambda_{2}}{2 L \lambda_{1}}$
\item $\lambda_{1}=1, \lambda_{2}=2$ and $\lambda_{3}=\frac{2 \sqrt{6} \sqrt{K (24 K+4 L)}+24 K+2 L}{2 L}$
\end{enumerate}

\

\item[(II)] For $\mathcal{W}_{1}\oplus \mathcal{W}_{2}$: Condition: $\lambda_3=-\lambda_1+\lambda_2$, with $\lambda_2>\lambda_1$.

\begin{eqnarray*}
2s_{1}-s&=&-\dfrac{5}{6}\|(dF)^{-}\|^{2}+\dfrac{1}{8}\|N^{0}\|^{2}+\dfrac{1}{2}\|(dF)^{+}\|^{2}\\
&=&-5K\dfrac{(2\lambda_{1}+\lambda_{2})^{2}}{36\lambda_{1}^{2}\lambda_{2}}+K\dfrac{(-2\lambda_{1}+\lambda_{1}+\lambda_{2})^{2}+(\lambda_{1}-2\lambda_{1}+\lambda_{2})^{2}+(2\lambda_{1}-2\lambda_{2})^{2}}{27\lambda_{1}^{2}\lambda_{2}}\\
&=&\frac{K \left(-4 \lambda_1^2-12 \lambda_1 \lambda_2+\lambda_2^2\right)}{12 \lambda_1^2 \lambda_2}
\end{eqnarray*}

We conclude that $2s_{1}-s=0$ if and only if $\lambda_2=(2 \sqrt{10} +6) \lambda_1$. 


\bigskip 

\item[(III)] For $\mathcal{W}_{1}\oplus \mathcal{W}_{3}$: Condition $\lambda_1=\lambda_2$.

\begin{eqnarray*}
2s_{1}-s&=&-\dfrac{5}{6}\|(dF)^{-}\|^{2}+\dfrac{1}{8}\|N^{0}\|^{2}+\dfrac{1}{2}\|(dF)^{+}\|^{2}\\
&=&-5K\dfrac{(2\lambda_{1}+\lambda_{2})^{2}}{36\lambda_{1}^{2}\lambda_{2}}+L\dfrac{(\lambda_{1}-\lambda_{2}+\lambda_{3})^{2}}{12\lambda_{1}\lambda_{2}\lambda_{3}}
\end{eqnarray*}.

In this case, we conclude that $2s_{1}-s=0$ if and only if $\lambda_1=\frac{L \lambda_3}{15 K}$.

\end{enumerate}

\

\item[(b)] {\bf Flag manifolds of type II:} 

We start by computing the components of the covariant derivative of the K\"ahler form. By a direct computation we have  
\begin{equation}
   \|(dF)^{-}\|^{2}=\|N^{0}\|^{2}=0, \quad {\rm and} \quad \|(dF)^{+}\|^{2}=L\dfrac{(\lambda_{1}-\lambda_{2}+\lambda_{3})^{2}}{6\lambda_{1}\lambda_{2}\lambda_{3}}
\end{equation}
where the non-zero constant $L$ is defined by equation \ref{constanteKL} and depends on the flag manifold. 

By equation \ref{formula-curvatura} we have
\begin{eqnarray*}
2s_{1}-s&=&-\dfrac{5}{6}\|(dF)^{-}\|^{2}+\dfrac{1}{8}\|N^{0}\|^{2}+\dfrac{1}{2}\|(dF)^{+}\|^{2}\\
&=&L\dfrac{(\lambda_{1}-\lambda_{2}+\lambda_{3})^{2}}{12\lambda_{1}\lambda_{2}\lambda_{3}}
\end{eqnarray*}

In this case, $2s_{1}-s=0$ if and only if $\lambda_{2}=\lambda_{1}+\lambda_{3}$, i.e., $(g,J_2)$ is Kähler.  
\end{enumerate}
\end{proof}

\begin{table}[h]
\caption{Invariant almost complex structure $J_3=(+,-,+)$}
\label{table-J3}
\begin{tabular}{|c|c|c|cc|}
\hline
\multirow{2}{*}{\begin{tabular}[c]{@{}c@{}}Flag\\ Manifold\end{tabular}} & \multirow{2}{*}{\begin{tabular}[c]{@{}c@{}}$(1,2)$-triples\\ $(\alpha+\beta+\gamma)=0$\end{tabular}} & \multirow{2}{*}{\begin{tabular}[c]{@{}c@{}}$(0,3)$-triples\\ $(\alpha+\beta+\gamma)=0$\end{tabular}} & \multicolumn{2}{c|}{\begin{tabular}[c]{@{}c@{}}Gray-Hervella\\ Classification\end{tabular}}                                                                                                                 \\ \cline{4-5} 
                                                                         &                                                                                                      &                                                                                                      & \multicolumn{1}{c|}{Class}                                                                                                 & Condition on the metric                                                        \\ \hline
Type I                                                                   & $\alpha \in \mathfrak{m}_1, \beta\in \mathfrak{m}_2, \gamma\in \mathfrak{m}_3$                                                        & $\alpha \in \mathfrak{m}_1, \beta\in \mathfrak{m}_1, \gamma\in \mathfrak{m}_2$                                                        & \multicolumn{1}{c|}{\begin{tabular}[c]{@{}c@{}}$\mathcal{W}_1\oplus \mathcal{W}_2$\\ $\mathcal{W}_1\oplus \mathcal{W}_3$\\ $\mathcal{W}_1\oplus \mathcal{W}_2\oplus \mathcal{W}_3$\end{tabular}} & \begin{tabular}[c]{@{}c@{}}$\lambda_3=-\lambda_1+\lambda_2$ and $\lambda_2>\lambda_1$\\ $\lambda_1=\lambda_2$\\ Otherwise\end{tabular} \\ \hline
Type II                                                                  & $\alpha \in \mathfrak{m}_1, \beta\in \mathfrak{m}_2, \gamma\in \mathfrak{m}_3$                                                        & \ding{55}                                                                                                 & \multicolumn{1}{c|}{\begin{tabular}[c]{@{}c@{}}K\"ahler\\ $\mathcal{W}_3$\end{tabular}}                                                & \begin{tabular}[c]{@{}c@{}}$\lambda_3=\lambda_1+\lambda_2$\\ Otherwise\end{tabular}                    \\ \hline
\end{tabular}
\end{table}

\begin{theorem} \label{teo3somandosJ3}
Let $G/K$ be a generalized flag manifold with three isotropy summands equipped with the invariant almost complex structure $J_3$ described in Table \ref{table-J3} and invariant Riemannian metric $g=(\lambda_1, \lambda_2,\lambda_3)$. Then $2s_1-s=0$ holds for $(G/K, g, J_3)$ if the metric $g$ satisfies the following conditions: 
\begin{enumerate}
    \item $G/K$ is type I
    \begin{enumerate}
        \item $\lambda_2=\lambda_1\left( 6-\frac{2L}{K}+\frac{2}{K}\sqrt{10K^2-5KL + L^2}\right)$, $\lambda_3=\lambda_1\left( 5-\frac{2L}{K}+\frac{2}{K}\sqrt{10K^2-5KL + L^2}\right) $
        \item $\lambda_1<\lambda_2$ and\\ $\lambda_3=\frac{1}{2} \left( \frac{1}{L\lambda_1}\sqrt{K \left(4 \lambda_1^2+12 \lambda_1 \lambda_2-\lambda_2^2\right) \left(K \left(4 \lambda_1^2+12 \lambda_1 \lambda_2-\lambda_2^2\right)+4 L \lambda_1 (\lambda_1-\lambda_2)\right)} +2 \lambda_1-2 \lambda_2\right)+ \frac{1}{2 L \lambda_1} K \left(4 \lambda_1^2+12 \lambda_1 \lambda_2-\lambda_2^2\right). $
        \item[] In the case of items $a)$ and $b)$ above,  $(g,J_3)$ belongs to the Gray-Hervella class $\mathcal{W}_{1}\oplus \mathcal{W}_{2}\oplus \mathcal{W}_{3}$.
        \item $\lambda_1=\lambda_2=\frac{L \lambda_3}{15K}$.  In this case, $(g,J_3)$ belongs to the Gray-Hervella class $\mathcal{W}_{1}\oplus \mathcal{W}_{3}$.
    \end{enumerate}
    \item $G/K$ is type II
    \begin{enumerate}
        \item $g=(\lambda_1,\lambda_2, \lambda_3)$ with $\lambda_1=\lambda_2+\lambda_3$. In this case, $(g,J_3)$ is K\"ahler. 
    \end{enumerate}
\end{enumerate}
\end{theorem}

\begin{proof}

We will proceed as in the proof of Theorem \ref{teo3somandosJ1}. Let us split our analysis for flag manifolds of type I and II:

\begin{enumerate}

\item[(a)] {\bf Flag manifolds of type I:} We obtain

$$\|(dF)^{-}\|^{2}=K\dfrac{(2\lambda_{1}+\lambda_{2})^{2}}{6\lambda_{1}^{2}\lambda_{2}}, \quad\quad \|(dF)^{+}\|^{2}=L\dfrac{(-\lambda_{1}+\lambda_{2}+\lambda_{3})^{2}}{6\lambda_{1}\lambda_{2}\lambda_{3}},$$

$$\|N^{0}\|^{2}=8K\dfrac{(-2\lambda_{1}+\lambda_{1}+\lambda_{2})^{2}+(\lambda_{1}-2\lambda_{1}+\lambda_{2})^{2}+(2\lambda_{1}-2\lambda_{2})^{2}}{27\lambda_{1}^{2}\lambda_{2}},$$
where the non-zero constants $K$ and $L$ are defined by equation \ref{constanteKL} and depend on the flag manifold. 

Let us analyze case-by-case the Gray-Hervella class with respect to the almost complex structure $J_3$: 

\

\begin{enumerate}
 
\item[(I)] For $\mathcal{W}_{1}\oplus \mathcal{W}_{2}\oplus \mathcal{W}_{3}$: 
\begin{eqnarray*}
2s_{1}-s&=&-\dfrac{5}{6}\|(dF)^{-}\|^{2}+\dfrac{1}{8}\|N^{0}\|^{2}+\dfrac{1}{2}\|(dF)^{+}\|^{2}\\
&=&-5K\dfrac{(2\lambda_{1}+\lambda_{2})^{2}}{36\lambda_{1}^{2}\lambda_{2}}+K\dfrac{(-2\lambda_{1}+\lambda_{1}+\lambda_{2})^{2}+(\lambda_{1}-2\lambda_{1}+\lambda_{2})^{2}+(2\lambda_{1}-2\lambda_{2})^{2}}{27\lambda_{1}^{2}\lambda_{2}}\\
&+&L\dfrac{(-\lambda_{1}+\lambda_{2}+\lambda_{3})^{2}}{12\lambda_{1}\lambda_{2}\lambda_{3}}
\end{eqnarray*}

Therefore we have $2s_1-s=0$ if, and only if  $$K (-4 \lambda_1^2 - 12 \lambda_1 \lambda_2 + \lambda_2^2) \lambda_3 + L \lambda_1 (-\lambda_1 + \lambda_2 + \lambda_3)^2=0.$$
 
 Examples of solutions:
 
 \begin{enumerate}
     \item $\lambda_2=\lambda_1\left( 6-\frac{2L}{K}+\frac{2}{K}\sqrt{10K^2-5KL + L^2}\right)$, $\lambda_3=\lambda_1\left( 5-\frac{2L}{K}+\frac{2}{K}\sqrt{10K^2-5KL + L^2}\right) $
     \item $\lambda_1>0, 0<\lambda_2<\lambda_1$ and\\ $\lambda_3=\frac{1}{2} \left( \frac{1}{L\lambda_1}\sqrt{K \left(4 \lambda_1^2+12 \lambda_1 \lambda_2-\lambda_2^2\right) \left(K \left(4 \lambda_1^2+12 \lambda_1 \lambda_2-\lambda_2^2\right)+4 L \lambda_1 (\lambda_1-\lambda_2)\right)} +2 \lambda_1-2 \lambda_2\right)$\\
     $+ \frac{ K \left(4 \lambda_1^2+12 \lambda_1 \lambda_2-\lambda_2^2\right)}{2 L \lambda_1} $
\end{enumerate}

\

\item[(II)] For $\mathcal{W}_{1}\oplus \mathcal{W}_{2}$: Condition: $\lambda_3=\lambda_1-\lambda_2$, with $\lambda_1>\lambda_2$.

\begin{eqnarray*}
2s_{1}-s&=&-\dfrac{5}{6}\|(dF)^{-}\|^{2}+\dfrac{1}{8}\|N^{0}\|^{2}+\dfrac{1}{2}\|(dF)^{+}\|^{2}\\
&=&-5K\dfrac{(2\lambda_{1}+\lambda_{2})^{2}}{36\lambda_{1}^{2}\lambda_{2}}+K\dfrac{(-2\lambda_{1}+\lambda_{1}+\lambda_{2})^{2}+(\lambda_{1}-2\lambda_{1}+\lambda_{2})^{2}+(2\lambda_{1}-2\lambda_{2})^{2}}{27\lambda_{1}^{2}\lambda_{2}}\\
&=&\frac{K \left(-4 \lambda_1^2-12 \lambda_1 \lambda_2+\lambda_2^2\right)}{12 \lambda_1^2 \lambda_2}
\end{eqnarray*}

There isn't solution of the equation $2s_{1}-s=0$. 

\item[(III)] For $\mathcal{W}_{1}\oplus \mathcal{W}_{3}$: Condition $\lambda_1=\lambda_2$

\begin{eqnarray*}
2s_{1}-s&=&-\dfrac{5}{6}\|(dF)^{-}\|^{2}+\dfrac{1}{8}\|N^{0}\|^{2}+\dfrac{1}{2}\|(dF)^{+}\|^{2}\\
&=&-5K\dfrac{(2\lambda_{1}+\lambda_{2})^{2}}{36\lambda_{1}^{2}\lambda_{2}}+L\dfrac{(-\lambda_{1}+\lambda_{2}+\lambda_{3})^{2}}{12\lambda_{1}\lambda_{2}\lambda_{3}}
\end{eqnarray*}

In this case, we conclude that $2s_{1}-s=0$ if and only if $\lambda_3=\frac{15 K \lambda_1}{L}$. 

\end{enumerate}

\

\item[(b)] {\bf Flag manifolds of type II: } 
We start by computing the components of the covariant derivative of the K\"ahler form. By a direct computation we have  
 \begin{equation}
   \|(dF)^{-}\|^{2}=\|N^{0}\|^{2}=0, \quad {\rm and} \quad \|(dF)^{+}\|^{2}=L\dfrac{(-\lambda_{1}+\lambda_{2}+\lambda_{3})^{2}}{6\lambda_{1}\lambda_{2}\lambda_{3}}.
\end{equation}
where the non-zero constant $L$ is defined by equation \ref{constanteKL} and depends on the flag manifold.  

By equation \ref{formula-curvatura} we have
\begin{eqnarray*}
2s_{1}-s&=&-\dfrac{5}{6}\|(dF)^{-}\|^{2}+\dfrac{1}{8}\|N^{0}\|^{2}+\dfrac{1}{2}\|(dF)^{+}\|^{2}\\
&=&L\dfrac{(-\lambda_{1}+\lambda_{2}+\lambda_{3})^{2}}{12\lambda_{1}\lambda_{2}\lambda_{3}}.
\end{eqnarray*}

Hence $2s_{1}-s=0$ if and only if $\lambda_{1}=\lambda_{2}+\lambda_{3}$, i.e., $(g, J_3)$ is K\"ahler. 

\end{enumerate}

\end{proof}

\begin{table}[h]
\caption{Invariant almost complex structure $J_4=(+,+,-)$}
\label{table-J4}
\begin{tabular}{|c|c|c|cc|}
\hline
\multirow{2}{*}{\begin{tabular}[c]{@{}c@{}}Flag\\ Manifold\end{tabular}} & \multirow{2}{*}{\begin{tabular}[c]{@{}c@{}}$(1,2)$-triples\\ $(\alpha+\beta+\gamma)=0$\end{tabular}} & \multirow{2}{*}{\begin{tabular}[c]{@{}c@{}}$(0,3)$-triples\\ $(\alpha+\beta+\gamma)=0$\end{tabular}} & \multicolumn{2}{c|}{\begin{tabular}[c]{@{}c@{}}Gray-Hervella\\ Classification\end{tabular}}                                                                                                       \\ \cline{4-5} 
                                                                         &                                                                                                      &                                                                                                      & \multicolumn{1}{c|}{Class}                                                                                                 & Condition on the metric                                              \\ \hline
Type I                                                                   & $\alpha \in \mathfrak{m}_1, \beta\in \mathfrak{m}_1, \gamma\in \mathfrak{m}_2$                                                        & $\alpha \in \mathfrak{m}_1, \beta\in \mathfrak{m}_2, \gamma\in \mathfrak{m}_3$                                                        & \multicolumn{1}{c|}{\begin{tabular}[c]{@{}c@{}}$\mathcal{W}_1\oplus \mathcal{W}_2$\\ $\mathcal{W}_1\oplus \mathcal{W}_3$\\ $\mathcal{W}_1\oplus \mathcal{W}_2\oplus \mathcal{W}_3$\end{tabular}} & \begin{tabular}[c]{@{}c@{}}$\lambda_2=2\lambda_1$\\ $\lambda_1=\lambda_2=\lambda_3$\\ Otherwise\end{tabular} \\ \hline
Type II                                                                  & \ding{55}                                                                                                 & $\alpha \in \mathfrak{m}_1, \beta\in \mathfrak{m}_2, \gamma\in \mathfrak{m}_3$                                                        & \multicolumn{1}{c|}{\begin{tabular}[c]{@{}c@{}}$\mathcal{W}_1$\\ $\mathcal{W}_1\oplus \mathcal{W}_2$\end{tabular}}                                       & \begin{tabular}[c]{@{}c@{}}$\lambda_1=\lambda_2=\lambda_3$\\ Otherwise\end{tabular}          \\ \hline
\end{tabular}
\end{table}

\begin{theorem} \label{teo3somandosJ4}
Let $G/K$ be a generalized flag manifold with three isotropy summands equipped with the invariant almost complex structure $J_4$ described in Table \ref{table-J4} and invariant Riemannian metric $g=(\lambda_1, \lambda_2,\lambda_3)$. Then $2s_1-s=0$ holds for $(G/K, g, J_4)$ if the metric $g$ satisfies the following conditions: 
\begin{enumerate}
    \item $G/K$ is type I
    \begin{enumerate}
        \item $g=(\lambda_1,\lambda_2, \lambda_3)$ with $\lambda_1=1$, $\lambda_2=1$ and $\lambda_3=\frac{1}{2} \sqrt{\frac{160 K^2-24 K L+L^2}{K^2}}+\frac{12 K-L}{2 K}$. In this case, $(g,J_4)$ belongs to the Gray-Hervella class $\mathcal{W}_{1}\oplus \mathcal{W}_{2}\oplus \mathcal{W}_{3}$.
        \item $g=(\lambda_1,2\lambda_1, (9+2\sqrt{22})\lambda_1)$. In this case, $(g,J_4)$ belongs to the Gray-Hervella class $\mathcal{W}_{1}\oplus \mathcal{W}_{2}$. 
    \end{enumerate}
    \item $G/K$ is type II
    \begin{enumerate}
        \item $g=(\lambda_1,\lambda_2, \lambda_3)$ with $\lambda_3=2 \sqrt{2} \sqrt{\lambda_1^2+3 \lambda_1 \lambda_2+\lambda_2^2}+3 (\lambda_1+\lambda_2)$. In this case, $(g,J_4)$ belongs to the Gray-Hervella class $\mathcal{W}_{1}\oplus \mathcal{W}_{2}$.
    \end{enumerate}
\end{enumerate}
\end{theorem}

\begin{proof}
Let us split our analysis for flag manifolds of type I and II:

\begin{enumerate}

\item[(a)] {\bf Flag manifolds of type I:} We get

$$\|(dF)^{-}\|^{2}=K\dfrac{(\lambda_{1}+\lambda_{2}+\lambda_{3})^{2}}{6\lambda_{1}\lambda_{2}\lambda_{3}}, \quad \|(dF)^{+}\|^{2}=L\dfrac{(-\lambda_{1}-\lambda_{1}+\lambda_{2})^{2}}{6\lambda_{1}^{2}\lambda_{2}}, $$

$$\|N^{0}\|^{2}=8K\dfrac{(-2\lambda_{1}+\lambda_{2}+\lambda_{3})^{2}+(\lambda_{1}-2\lambda_{2}+\lambda_{3})^{2}+(\lambda_{1}+\lambda_{2}-2\lambda_{3})^{2}}{27\lambda_{1}\lambda_{2}\lambda_{3}}.$$
where the non-zero constants $K$ and $L$ are defined by equation \ref{constanteKL} and depends on the flag manifold. 


Let us analyze case-by-case the Gray-Hervella class with respect to the almost complex structure $J_4$: 
 
\begin{enumerate}
 
\item[(I)] For $\mathcal{W}_{1}\oplus \mathcal{W}_{2}\oplus \mathcal{W}_{3}$: 
\begin{eqnarray*}
2s_{1}-s&=&-\dfrac{5}{6}\|(dF)^{-}\|^{2}+\dfrac{1}{8}\|N^{0}\|^{2}+\dfrac{1}{2}\|(dF)^{+}\|^{2}\\
&=&-5K\dfrac{(\lambda_{1}+\lambda_{2}+\lambda_{3})^{2}}{36\lambda_{1}\lambda_{2}\lambda_{3}}+K\dfrac{(-2\lambda_{1}+\lambda_{2}+\lambda_{3})^{2}+(\lambda_{1}-2\lambda_{2}+\lambda_{3})^{2}+(\lambda_{1}+\lambda_{2}-2\lambda_{3})^{2}}{27\lambda_{1}\lambda_{2}\lambda_{3}}\\
&+&L\dfrac{(-\lambda_{1}-\lambda_{1}+\lambda_{2})^{2}}{12\lambda_{1}^{2}\lambda_{2}}
\end{eqnarray*}

Therefore, $2s_1-s=0$ if, and only if,
$$
K \lambda_1 \left(\lambda_1^2-6 \lambda_1 (\lambda_2+\lambda_3)+\lambda_2^2-6 \lambda_2 \lambda_3+\lambda_3^2\right)+L \lambda_3 (\lambda_2-2 \lambda_1)^2=0.
$$
 
Note that $\lambda_1=1$, $\lambda_2=1$ and $\lambda_3=\frac{1}{2} \sqrt{\frac{160 K^2-24 K L+L^2}{K^2}}+\frac{12 K-L}{2 K}$ is a solution of the equation
$2s_{1}-s=0$. 

\
 
\item[(II)] For $\mathcal{W}_{1}\oplus \mathcal{W}_{2}$: Condition: $\lambda_2=2\lambda_1$ 

\begin{eqnarray*}
2s_{1}-s&=&-\dfrac{5}{6}\|(dF)^{-}\|^{2}+\dfrac{1}{8}\|N^{0}\|^{2}+\dfrac{1}{2}\|(dF)^{+}\|^{2}\\
&=&-5K\dfrac{(\lambda_{1}+\lambda_{2}+\lambda_{3})^{2}}{36\lambda_{1}\lambda_{2}\lambda_{3}}+K\dfrac{(-2\lambda_{1}+\lambda_{2}+\lambda_{3})^{2}+(\lambda_{1}-2\lambda_{2}+\lambda_{3})^{2}+(\lambda_{1}+\lambda_{2}-2\lambda_{3})^{2}}{27\lambda_{1}\lambda_{2}\lambda_{3}}
\end{eqnarray*}

Therefore, $2s_1-s=0$ if, and only if,
$$K \left(5 x\lambda_1^2-12 \lambda_1 \lambda_3-6 \lambda_1 (2 \lambda_1+\lambda_3)+\lambda_3^2\right)=0$$

In this case, we conclude that $2s_{1}-s=0$ if and only if the metric satisfies $(\lambda_1,2\lambda_1, (9+2\sqrt{22})\lambda_1) $  

\

\item[(III)] For $\mathcal{W}_{1}\oplus \mathcal{W}_{3}$: Condition: $\lambda_{1}=\lambda_{2}=\lambda_{3}$

\begin{eqnarray*}
2s_{1}-s&=&-\dfrac{5}{6}\|(dF)^{-}\|^{2}+\dfrac{1}{8}\|N^{0}\|^{2}+\dfrac{1}{2}\|(dF)^{+}\|^{2}\\
&=&-5K\dfrac{(\lambda_{1}+\lambda_{2}+\lambda_{3})^{2}}{36\lambda_{1}\lambda_{2}\lambda_{3}}+
L\dfrac{(-\lambda_{1}-\lambda_{1}+\lambda_{2})^{2}}{12\lambda_{1}^{2}\lambda_{2}}
\end{eqnarray*}

Therefore, $2s_1-s=0$ if, and only if,
$$
-45 K \lambda_3^3 + 3 L \lambda_3^3=0, 
$$
and there is not solution of the equation $2s_{1}-s=0$ independent of $L$ and $K$. 
{\bf Remark: }It is worth to point out that $L=15K$ is a solution of the equation $2s_{1}-s=0$, but it solution depends on case-by-case examination and we do not include it as a general solution. 
\end{enumerate}

\

\item[(b)] {\bf Flag manifolds type II:} We find 
$$ \|(dF)^{+}\|^{2}=0, \quad\quad \|(dF)^{-}\|^{2}=K\dfrac{(\lambda_{1}+\lambda_{2}+\lambda_{3})^{2}}{6\lambda_{1}\lambda_{2}\lambda_{3}}, $$

$$\|N^{0}\|^{2}=8K\dfrac{(-2\lambda_{1}+\lambda_{2}+\lambda_{3})^{2}+(\lambda_{1}-2\lambda_{2}+\lambda_{3})^{2}+(\lambda_{1}+\lambda_{2}-2\lambda_{3})^{2}}{27\lambda_{1}\lambda_{2}\lambda_{3}}$$
where the non-zero constant $K$ is defined by equation \ref{constanteKL} and depends on the flag manifold.  

By equation \ref{formula-curvatura} we obtain
\begin{eqnarray*}
2s_{1}-s&=&-\dfrac{5}{6}\|(dF)^{-}\|^{2}+\dfrac{1}{8}\|N^{0}\|^{2}+\dfrac{1}{2}\|(dF)^{+}\|^{2}\\
&=&-5K\dfrac{(\lambda_{1}+\lambda_{2}+\lambda_{3})^{2}}{36\lambda_{1}\lambda_{2}\lambda_{3}}+K\dfrac{(-2\lambda_{1}+\lambda_{2}+\lambda_{3})^{2}+(\lambda_{1}-2\lambda_{2}+\lambda_{3})^{2}+(\lambda_{1}+\lambda_{2}-2\lambda_{3})^{2}}{27\lambda_{1}\lambda_{2}\lambda_{3}}
\end{eqnarray*}

Therefore, $2s_1-s=0$ if, and only if, 
$$
K \left(\lambda_1^2-6 \lambda_1 (\lambda_2+\lambda_3)+\lambda_2^2-6 \lambda_2 \lambda_3+\lambda_3^2\right)=0.
$$
and the positive solutions are given by $\lambda_1>0$, $\lambda_2>0$ and $$
\lambda_3=2 \sqrt{2} \sqrt{\lambda_1^2+3 \lambda_1 \lambda_2+\lambda_2^2}+3 (\lambda_1+\lambda_2).
$$
\end{enumerate}
\end{proof}

\section{Examples}
In this section we will provide some explicit computations on some generalized flag manifolds. Here one can obtain precisely the $(1,2)$ and $(0,3)$ triples and consequently the constants $K$ and $L$ defined by equation \ref{constanteKL}. Using this facts we will provide some solutions for the equation $s=2s_1$ for specific flag manifolds. 

\subsection{$G_{2}/U(2)$ } 
We consider the Type I flag manifold $G_{2}/U(2)$ of real dimension $10$, where $U(2)$ is represented by the long root of $G_2$ (the algebraic structure of the $G_{2}$ can be seen in \cite{gdois}). This manifold has three isotropic summands given by
$\got{m}_{1}$, $\got{m}_{2}$ and $\got{m}_{3}$ given by

$$
\begin{array}{l}
\got{m}_{1}=\got{g}_{\alpha_{1}+\alpha_{2}}\oplus \got{g}_{\alpha_{2}},\\
\got{m}_{2}=\got{g}_{\alpha_{1}+2\alpha_{2}},\\
\got{m}_{3}=\got{g}_{\alpha_{1}+3\alpha_{2}}\oplus \got{g}_{2\alpha_{1}+3\alpha_{2}}.
\end{array}
$$
From proposition \ref{numest}, $G_{2}/U(2)$ admits $4$ invariant almost complex structures, up to conjugation:
$$
\begin{array}{ccccccc}
J_{1}=(+,+,+)&&&&&&J_{3}=(-,+,-)\\
J_{2}=(-,+,+)&&&&&&J_{4}=(+,+,-),
\end{array}
$$
where $J_{1}$ is integrable.

The triples of zero sum of $G_{2}/U(2)$ are (up to sign):
$$
\begin{array}{cl}
1) &\alpha_{2}+(\alpha_{1}+\alpha_{2})-(\alpha_{1}+2\alpha_{2})=0\\
2) & \alpha_{2}+(\alpha_{1}+2\alpha_{2})-(\alpha_{1}+3\alpha_{2})=0\\
3) & (\alpha_{1}+\alpha_{2})+(\alpha_{1}+2\alpha_{2})-(2\alpha_{1}+3\alpha_{2})=0.
\end{array}
$$

 We also compute the value of the constants $K$ and $L$ (equations \ref{constanteKL}) for each specific invariant almost complex structure. We summarize the computation in the Table \ref{KLparaG2}.

\begin{table}[h]
\centering

\caption{Constants $K$ and $L$} \label{KLparaG2}

\begin{tabular}{|c|c|c|}
\hline
Structure & $K$ & $L$\\
\hline
\hline
$(+,+,+)$ & $0$  &  $5/12$ \\
\hline
$(-,+,+)$ & $1/6$  &  $1/4$ \\
\hline
$(+,-,+)$ & $1/6$ & $1/4$ \\
\hline
$(+,+,-)$ & $1/4$ & $1/6$ \\
\hline
\end{tabular}
\end{table}

Consider the orthonormal basis of $TM_{\mathbb{C}}$:
\begin{eqnarray*}
\left\{\dfrac{X_{\alpha_{2}}}{\sqrt{x}},\dfrac{X_{\alpha_{1}+\alpha_{2}}}{\sqrt{x}},\dfrac{X_{\alpha_{1}+2\alpha_{2}}}{\sqrt{y}},\dfrac{X_{\alpha_{1}+3\alpha_{2}}}{\sqrt{z}},\dfrac{X_{2\alpha_{1}+3\alpha_{2}}}{\sqrt{z}},
\dfrac{X_{-\alpha_{2}}}{\sqrt{x}},\right.\\
\left.\dfrac{X_{-(\alpha_{1}+\alpha_{2})}}{\sqrt{x}},\dfrac{X_{-(\alpha_{1}+2\alpha_{2})}}{\sqrt{y}},\dfrac{X_{-(\alpha_{1}+3\alpha_{2})}}{\sqrt{z}},\dfrac{X_{-(2\alpha_{1}+3\alpha_{2})}}{\sqrt{z}}\right\}.
\end{eqnarray*}

One can classified Gray-Hervella (GH) classes for $G_2/U(2)$ and study the equation $2s_1=s$ for each GH-class. 

\begin{itemize}
\item Let us consider the invariant almost complex structure $J_{1}=(+,+,+)$ (integrable) and Riemannian metric parameterized by $g=(x,y,z)$. We find


$$
\|(dF)^{-}\|^{2}=\|N^{0}\|^{2}=0,
$$
and
$$
\|(dF)^{+}\|=\frac{1}{36 x^2 y}\left(\frac{3 x (x+y-z)^2}{z}+2 (y-2 x)^2 \right).
$$
The covariant derivative of Kähler form is 
\begin{eqnarray*}
\|DF\|^{2}=\frac{3 x^3+2 x^2 (3 y+z)+x \left(3 y^2-14 y z+3 z^2\right)+2 y^2 z}{36 x^2 y z}.
\end{eqnarray*}

We have $(J_{1},g)\in \mathcal{W}_{3}$ unless $y=2x$ and $z=3x$ which corresponds to the K\"ahler class. The solution for the equation $2s_{1}-s=0$ is $y=2x$ and $z=3x$, that is, the unique solution is given by the K\"ahler metric.

\ 

\item Let us consider the invariant almost complex structure $J_{2}=(-,+,+)$ (non-integrable) and Riemannian metric parameterized by $g=(x,y,z)$. We get
$$
\|(dF)^{-}\|^{2}=\dfrac{(2 x+y)^2}{18 x^2 y}, \quad\quad \|N^{0}\|^{2}=\frac{16 (x-y)^2}{27 x^2 y}, \quad\quad \|(dF)^{+}\|=\frac{(x-y+z)^2}{12 x y z}.
$$

The covariant derivative of the K\"ahler form is 
$$
\|DF\|^{2}=\frac{3 x^3+x^2 (14 z-6 y)+x \left(3 y^2-14 y z+3 z^2\right)+6 y^2 z}{36 x^2 y z}.
$$

\begin{enumerate}
    \item If the metric $g$ satisfies $x=y$ then $(J_{2},g)\in \mathcal{W}_{1}\oplus \mathcal{W}_{3}$. In this case
$$2s_{1}-s=0\Leftrightarrow z=10x.$$

\item If the metric $g$ satisfies  $z=-x+y$ and $y>x$ then $(J_{2},g)\in \mathcal{W}_{1}\oplus \mathcal{W}_{2}$. In this case
$$2s_{1}-s=0 \Leftrightarrow y=(2 \sqrt{10} +6)x.$$

\item Otherwise the pair $(J_{2},g)$ belongs to $\mathcal{W}_{1}\oplus \mathcal{W}_{2}\oplus \mathcal{W}_{3}$. In this case
$$
2s_{1}-s=\frac{3 x^3-2 x^2 (3 y+z)+3 x \left(y^2-10 y z+z^2\right)+2 y^2 z}{72 x^2 y z}
$$
The triples $x=1$, $y=2$ and $z=9-\sqrt{5}$ is an example of solution of the equation $2s_{1}-s=0$. 

\end{enumerate}

\bigskip

\item Let us consider the invariant almost complex structure $J_{3}=(+,-,+)$ (non-integrable) and Riemannian metric parameterized by $g=(x,y,z)$. We have

$$
\|(dF)^{-}\|^{2}=\frac{(2 x+y)^2}{18 x^2 y}, \quad \quad \|N^{0}\|^{2}=\frac{16 (x-y)^2}{27 x^2 y}, \quad \quad \|(dF)^{+}\|=\frac{(-x+y+z)^2}{12 x y z}.
$$

The covariant derivative of K\"ahler form is 
$$
\|DF\|^{2}=\frac{3 x^3+x^2 (2 z-6 y)+x \left(3 y^2-2 y z+3 z^2\right)+6 y^2 z}{36 x^2 y z}.
$$

\begin{enumerate}
    \item If the metric $g$ satisfies $x=y$ then $(J_{3},g)\in \mathcal{W}_{1}\oplus \mathcal{W}_{3}$. In this case
$$2s_{1}-s=0\Leftrightarrow z=10x.$$
    \item If the metric $g$ satisfies  $z=x-y$ and $x>y$ then $(J_{3},g)\in \mathcal{W}_{1}\oplus \mathcal{W}_{2}$. In this case there is no solution for equation $2s_{1}-s=0$. 
    \item Otherwise the pair $(J_{3},g)$ belongs to $\mathcal{W}_{1}\oplus \mathcal{W}_{2}\oplus \mathcal{W}_{3}$. One can provide an example of non-trivial solution of the equation $2s_{1}-s=0$ given by  $x=1$, $y=2$ and $z=7+4\sqrt{3}$.  
\end{enumerate}

\bigskip

\item Let us consider the invariant almost complex structure $J_{4}=(+,+,-)$ (non-integrable) and Riemannian metric parameterized by $g=(x,y,z)$.

The $(0,3)+(3,0)$-part exterior derivative of the K\"ahler form is given by
\begin{eqnarray*}
\|(dF)^{-}\|^{2}=\dfrac{(x+y+z)^2}{12 x y z} \ \ \mbox{and}
\end{eqnarray*}
$N^{0}$ is defined by
\begin{eqnarray*}
\|N^{0}\|^{2}=\dfrac{8 \left(x^2-x (y+z)+y^2-y z+z^2\right)}{9 x y z}.
\end{eqnarray*}
Moreover, the $(1,2)+(2,1)$-part of the K\"ahler form is
\begin{eqnarray*}
\|(dF)^{+}\|^{2}=\dfrac{(y-2 x)^2}{18 x^2 y}.
\end{eqnarray*}

The covariant derivative of the K\"ahler form is 
\begin{eqnarray*}
\|DF\|^{2}=\frac{9 x^3+x^2 (2 z-6 y)+x \left(9 y^2-14 y z+9 z^2\right)+2 y^2 z}{36 x^2 y z}.
\end{eqnarray*}

Note that $\|DF\|^{2}=\|(dF)^{+}\|^{2}+\dfrac{1}{4}\|N^{0}\|^{2}+\dfrac{1}{3}\|(dF)^{-}\|^{2}$.

\begin{enumerate}

\item The metric $g$ satisfies $x=y=z$. In this case $(J_{4},g)$ belongs to $\mathcal{W}_{1}\oplus \mathcal{W}_{3}$.

In this case there is not solution for equation $2s_{1}-s=0$. 

\item The metric $g$ satisfies the condition $y=2x$. In this case $(J_{4},g)$ belongs to $\mathcal{W}_{1}\oplus \mathcal{W}_{2}$. We obtain
$$2s_{1}-s=0\Leftrightarrow z=\dfrac{9y}{2}+\sqrt{22}y.$$

\

\item Otherwise the pair $(J_{4},g)$ belongs to $\mathcal{W}_{1}\oplus \mathcal{W}_{2}\oplus \mathcal{W}_{3}$. We get
\begin{eqnarray}\label{eqG2}
2s_{1}-s=\frac{3 x^3-2 x^2 (9 y+5 z)+x \left(3 y^2-26 y z+3 z^2\right)+2 y^2 z}{72 x^2 y z}.
\end{eqnarray}

One example of solution of Equation \ref{eqG2} is given by $x = 1$, $y = 1$ and\\ $z = \frac{1}{3} \left(5 \sqrt{13}+17\right)$ and therefore this metric satisfies the equation $2s_{1}-s=0$. 
\end{enumerate}

\end{itemize}

\subsection{$SU(n+2)/S(U(n)\times U(1)\times U(1)$)}

Let us consider the Cartan subalgebra $\got{h}$ of $\got{su}(n+2)$ described by
$$
\got{h}=\{\mbox{diag}(x_{1},\cdots,x_{n},x_{n+1},x_{n+2}):x_{1}+\cdots+x_{n}+x_{n+1}+x_{n+2}=0,\ x_{i}\in\mathbb{C}, \ i=1,\cdots, n\}.
$$
Moreover, the isotropy representation of $\mathbb{F}_{n+2}$ admits three isotropic summands:\\
$\mathfrak{m}=\mathfrak{m}_{1}\oplus \mathfrak{m}_{2} \oplus \mathfrak{m}_{3}$,
where
$$
\begin{array}{l}
\mathfrak{m}_{1}=\displaystyle\bigoplus_{i=1}^{n}\mathfrak{u}_{\alpha_{i,n+1}}, \quad\quad  
\mathfrak{m}_{2}=\mathfrak{u}_{\alpha_{n+1,n+2}}, \quad\quad
\mathfrak{m}_{3}=\displaystyle\bigoplus_{i=1}^{n}\mathfrak{u}_{\alpha_{i,n+2}}, \\ 
\end{array}
$$
with $\got{u}_{ij}=\got{su}(n+2)\cap (\got{g}_{ij}\oplus \got{g}_{ji})$.

The matrix representation of $\mathfrak{m}$ is:
$$
\left(
\begin{array}{cccc|c|c}
x_{1}&&&&*&*\\
&x_{2}&&&*&*\\
&&\ddots&&\vdots&\vdots\\
&&&x_{n}&*&*\\
\hline
&&&&x_{n+1}&*\\
\hline
&&&&&x_{n+2}
\end{array}
\right).
$$
The flag manifold $SU(n+2)/S(U(n)\times U(1)\times U(1))$ admits $3$ invariant almost complex structures, up to conjugation and equivalence:
$$
\begin{array}{ccccccccccc}
J_{1}=(+,+,+) &&&&& J_{2}=(-,+,+) &&&&& J_{3}=(+,+,-).
\end{array}
$$
where $J_{1}$ and $J_{2}$ are integrable.

It is easy to see the triples of roots are given by:
$$
\mbox{2n triples}:\left\{
\begin{array}{l}
\alpha_{1,n+1}+\alpha_{n+1,n+2}+\alpha_{n+2,1}=0\\
\alpha_{2,n+1}+\alpha_{n+1,n+2}+\alpha_{n+2,2}=0\\
\ \ \ \ \  \ \ \vdots\\
\alpha_{n,n+1}+\alpha_{n+1,n+2}+\alpha_{n+2,n}=0\\
-\alpha_{1,n+1}-\alpha_{n+1,n+2}-\alpha_{n+2,1}=0\\
-\alpha_{2,n+1}-\alpha_{n+1,n+2}-\alpha_{n+2,2}=0\\
\ \ \ \ \  \ \ \vdots\\
-\alpha_{n,n+1}-\alpha_{n+1,n+2}-\alpha_{n+2,n}=0\\
\end{array}
\right.
$$

We can also classify the invariant almost Hermitian structures as follow:
\begin{itemize}
\item $J_{1}=(+,+,+)$ integrable and $g=(x,y,z)$ invariant Riemannian metric: 
If $z=x+y$ then $(J_{1},g)$ is Kähler, otherwise $(J_{1},g)\in \mathcal{W}_{3}$. We have
$$ \|(dF)^{-}\|^{2}=\|N^{0}\|^{2}=0, \quad\quad \|(dF)^{+}\|^{2}=\dfrac{n(x+y-z)}{3x y z},$$
and $2s_{1}-s=0$ if and only if $z=x+y$, i.e., $(J_{1},g)$ is K\"ahler. 

\bigskip

\item $J_{2}=(-,+,+)$ integrable and $g=(x,y,z)$ invariant Riemannian metric: if $z=-x+y$ with $y>x$ then $(J_{2},g)$ is K\"ahler, otherwise $(J_{2},g)\in \mathcal{W}_{3}$. We have
$$
\|(dF)^{-}\|^{2}=\|N^{0}\|^{2}=0 \quad\quad \|(dF)^{+}\|^{2}=\dfrac{n(x-y+z)}{3x y z}
$$
and $2s_{1}-s=0$ if and only if $z=-x+y$, i.e., $(J_{2},g)$ is K\"ahler.

\bigskip

\item $J_{3}=(+,+,-)$ non-integrable and $g=(x,y,z)$ invariant Riemannian metric: if $x=y=z$ then $(J_{3},g)\in \mathcal{W}_{1}$,  otherwise $(J_{3},g)\in\mathcal{W}_{1}\oplus \mathcal{W}_{2}$. We have
$$
\|(dF)^{-}\|^{2}=\frac{n (x+y+z)^2}{3 x y z}, \quad \quad \|N^{0}\|^{2}=\frac{32 n \left(x^2-x (y+z)+y^2-y z+z^2\right)}{9 x y z}, \quad \quad \|(dF)^{+}\|^{2}=0.
$$
If $(J_{3},g)\in \mathcal{W}_{1}$ there is no solution of the equation $2s_{1}-s=0$. If $(J_{3},g)\in \mathcal{W}_{1}\oplus \mathcal{W}_{2}$ we find the following non-trivial solutions (non-K\"ahler) for the equation $2s_{1}-s=0$: $x>0, y>0$ and $z=3 (x+y) + 2 \sqrt{2} \sqrt{x^2+3 x y+y^2}$ ; or 
$x> 0, y> (3+2\sqrt{2})x$ and $z=3 (x+y)-2 \sqrt{2} \sqrt{x^2+3 x y+y^2}$.

 \end{itemize}

\subsection{$F_{4}/SU(3)\times SU(2)\times U(1)$} Let $\{\alpha_1, \alpha_2,\alpha_3, \alpha_4 \}$ be the simple roots of the Lie algebra of $F_4$, with maximal root $\mu=2\alpha_1+3\alpha_2+4\alpha_3+2\alpha_4$. The description of the components $\mathfrak{m}_1$, $\mathfrak{m}_2$, $\mathfrak{m}_3$ of the isotropy representation is given in Table \ref{coordF4}. The coordinates in the table are given in terms of the simple roots $\alpha_1, \ldots, \alpha_4$. For instance $[1,2,4,2]$ means $\alpha_1+2\alpha_2+4\alpha_3+2\alpha_2$. The triples of roots with zero sum are listed in Table \ref{tripleF4}. We also compute the value of the constants $K$ and $L$ (equations \ref{constanteKL}) for each specific invariant almost complex structure. We summarize the computation in the Table \ref{KLparaF4}.

Let us analyze the solution for equation $2s_1-s=0$ for each invariant almost Hermitian structure $(g,J)$, being $g$ an invariant Riemannian metric parametrized by $g=(x,y,z)$. 

\begin{table}[h!]
\centering

\caption{Positive roots of G/K in coordinates $\alpha_{1},\alpha_{2},\alpha_{3},\alpha_{4}$} \label{coordF4} 

\begin{tabular}{|c|c|c|}
\hline
Roots in $\got{m}_{1}$ & Roots  in $\got{m}_{2}$ & Roots  in $\got{m}_{3}$\\
\hline
\hline
[0,1,0,0] & [1,2,2,0] & [1,3,4,2]\\
\hline
[1,1,0,0] & [1,2,2,1] & [2,3,4,2] \\
\hline
[0,1,1,0] & [1,2,3,1] & \\
\hline
[1,1,1,0] & [1,2,2,2] & \\
\hline
[0,1,2,0] & [1,2,3,2] & \\
\hline
[0,1,1,1] & [1,2,4,2] & \\
\hline
[1,1,2,0] & & \\
\hline
[1,1,1,1] & & \\
\hline
[0,1,2,1] & & \\
\hline
[1,1,2,1] & & \\
\hline
[0,1,2,2] & & \\
\hline
[1,1,2,2] & & \\
\hline
\end{tabular}
\end{table}

\begin{table}[h]
\centering
\caption{Triples with zero sum: $\alpha+\beta+\gamma=0$}
\label{tripleF4}
\begin{tabular}{|c|c|c|c|}
\hline
$\alpha\in \got{m}_{1}$ & $\beta \in\got{m}_{1}$ & $-\gamma\in\got{m}_{2}$ & $m_{\alpha,\beta}^{2}$\\
\hline
\hline
[0,1,0,0] & [1,1,2,0] & [1,2,2,0] & $1$ \\ 
\hline
[0,1,0,0] & [1,1,2,1] & [1,2,2,1] & $1$\\
\hline
[0,1,0,0] & [1,1,2,2] & [1,2,2,2] & $1$\\
\hline
[1,1,0,0] & [0,1,2,0] & [1,2,2,0]& $1$\\
\hline
[1,1,0,0]& [0,1,2,1]&[1,2,2,1] & $1$\\
\hline
[1,1,0,0]&[0,1,2,2]&[1,2,2,2] & $1$ \\
\hline
[0,1,1,0]&[1,1,1,0]&[1,2,2,0] & $1$\\
\hline
[0,1,1,0]&[1,1,1,1]&[1,2,2,1] & $1/2$\\
\hline
[0,1,1,0] &
[1,1,2,1] &
[1,2,3,1] & $1/2$\\
\hline
[0,1,1,0] &
[1,1,2,2] &
[1,2,3,2] & $1$\\
\hline
[1,1,1,0] &
[0,1,1,1] &
[1,2,2,1] & $1/2$\\
\hline
[1,1,1,0] &
[0,1,2,1] &
[1,2,3,1] & $1/2$\\
\hline
[1,1,1,0] &
[0,1,2,2] &
[1,2,3,2] & $1$\\
\hline
[0,1,2,0] &
[1,1,1,1] &
[1,2,3,1] & $1$\\
\hline
[0,1,2,0] &
[1,1,2,2] &
[1,2,4,2] & $1$\\
\hline
[0,1,1,1] &
[1,1,2,0] &
[1,2,3,1] & $1$\\
\hline
[0,1,1,1] &
[1,1,1,1] &
[1,2,2,2] & $1$\\
\hline
[0,1,1,1] &
[1,1,2,1] &
[1,2,3,2] & $1/2$\\
\hline
[1,1,2,0] &
[0,1,2,2] &
[1,2,4,2] & $1$\\
\hline
[1,1,1,1] &
[0,1,2,1] &
[1,2,3,2] & $1/2$\\
\hline
[0,1,2,1] &
[1,1,2,1] &
[1,2,4,2] & $1$\\
\hline
-- & -- & -- & -- \\
\hline
$\alpha \in\got{m}_{1}$ & $\beta\in \got{m}_{2}$ & $-\gamma\in\got{m}_{3}$ & $m_{\alpha,\beta}^{2}$\\
\hline
[0,1,0,0] &
[1,2,4,2] &
[1,3,4,2] & $1$\\
\hline
[1,1,0,0] &
[1,2,4,2] &
[2,3,4,2] & $1$\\
\hline
[0,1,1,0] &
[1,2,3,2] &
[1,3,4,2] & $1$\\
\hline
[1,1,1,0] &
[1,2,3,2] &
[2,3,4,2] & $1$\\
\hline
[0,1,2,0] &
[1,2,2,2] &
[1,3,4,2] & $1$\\
\hline
[0,1,1,1] &
[1,2,3,1] &
[1,3,4,2] & $1$\\
\hline
[1,1,2,0] &
[1,2,2,2] &
[2,3,4,2] & $1$\\
\hline
[1,1,1,1] &
[1,2,3,1] &
[2,3,4,2] & $1$\\
\hline
[0,1,2,1] &
[1,2,2,1] &
[1,3,4,2] & $1$\\
\hline
[1,1,2,1] &
[1,2,2,1] &
[2,3,4,2] & $1$\\
\hline
[0,1,2,2] &
[1,2,2,0] &
[1,3,4,2] & $1$\\
\hline
[1,1,2,2] &
[1,2,2,0] &
[2,3,4,2] & $1$\\
\hline
\end{tabular}
\end{table}

\begin{table}[h]
\centering

\caption{Constants $K$ and $L$} \label{KLparaF4}

\begin{tabular}{|c|c|c|}
\hline
Structure & $K$ & $L$\\
\hline
\hline
$(+,+,+)$ & $0$  &  $30$ \\
\hline
$(-,+,+)$ & $18$  &  $12$ \\
\hline
$(+,-,+)$ & $18$ & $12$ \\
\hline
$(+,+,-)$ & $12$ & $18$ \\
\hline
\end{tabular}
\end{table}

\begin{itemize}

\item {\bf Invariant complex structure $J=(+,+,+)$ (integrable)}: if $y=2x$ and $z=3x$ the pair $(g,J)$ is K\"ahler, otherwise $(J,g)\in \mathcal{W}_{3}$. We have

$$\|(dF)^{-}\|^{2}=\|N^{0}\|^{2}=0, \quad \|(dF)^{+}\|^{2}=\frac{2}{x^2 y} \left(\frac{2 x (x+y-z)^2}{z}+3 (y-2 x)^2\right),$$ 
and $2s_{1}-s=0$ if and only if $y=2x$ and $z=3x$, i.e., when $(J,g)$ is Kähler.

\bigskip

\item {\bf Invariant almost complex structure $J_{1}=(-,+,+)$: } if $z=-x+y$ and $y>x$ then $(J_1,g)\in \mathcal{W}_{1}\oplus \mathcal{W}_{2}$, if $x=y$ then  $(J_1,g)\in\mathcal{W}_{1}\oplus \mathcal{W}_{3}$, otherwise $(J_1,g)\in \mathcal{W}_{1}\oplus \mathcal{W}_{2}\oplus \mathcal{W}_{3}$. We obtain
$$\|(dF)^{-}\|^{2}=\frac{6 (2 x+y)^2}{x^2 y}, \quad \|N^{0}\|^{2}=\frac{64 (x-y)^2}{x^2 y}, \quad \|(dF)^{+}\|^{2}=\frac{4 (x-y+z)^2}{x y z}.$$
If $(J_1,g)\in \mathcal{W}_{1}\oplus \mathcal{W}_{2}$ we have solution for the equation $2s_1-s=0$ given by the metric $y=(2\sqrt{10}+6)x$ and $x>0$. For $(J_2,g)\in \mathcal{W}_{1}\oplus \mathcal{W}_{3}$ we get the metric given by $x=y>0$, $z=45y/2$ is a solution for the equation $2s_1-s=0$. If $(J_1,g)\in \mathcal{W}_{1}\oplus \mathcal{W}_{2}\oplus \mathcal{W}_{3}$ then we have the families of solutions of $2s_1-s=0$:
\begin{enumerate}
    \item $g=(x,y,z)$ parameterized by $x>0, 0<y<x$ and\\ $z=\frac{8 x^2+\sqrt{3} \sqrt{16 x^4+224 x^3 y+512 x^2 y^2-80 x y^3+3 y^4}+40 x y-3 y^2}{4 x}$. 
    \item $g=(x,y,z)$ parameterized by $x=1$, $y=2$ and $z=19+6\sqrt{10}$.
\end{enumerate}

\bigskip

\item {\bf Invariant almost complex structure $J_{2}=(+,-,+)$: } if $z=x-y$ and $x>y$ then $(J_2,g)\in \mathcal{W}_{1}\oplus \mathcal{W}_{2}$, if $x=y$ then  $(J_2,g)\in\mathcal{W}_{1}\oplus \mathcal{W}_{3}$, otherwise $(J_2,g)\in \mathcal{W}_{1}\oplus \mathcal{W}_{2}\oplus \mathcal{W}_{3}$. We find
$$\|(dF)^{-}\|^{2}=\frac{6 (2 x+y)^2}{x^2 y}, \quad \|N^{0}\|^{2}=\frac{64 (x-y)^2}{x^2 y}, \quad \|(dF)^{+}\|^{2}=\frac{4 (-x+y+z)^2}{x y z}.$$
If $(J_2,g)\in \mathcal{W}_{1}\oplus \mathcal{W}_{2}$ we do not have solution for the equation $2s_1-s=0$. For $(J_2,g)\in \mathcal{W}_{1}\oplus \mathcal{W}_{3}$ we have the metric given by $x=y>0$, $z=45y/2$ is a solution for the equation $2s_1-s=0$. If $(J_2,g)\in \mathcal{W}_{1}\oplus \mathcal{W}_{2}\oplus \mathcal{W}_{3}$ then we have the families of solutions of $2s_1-s=0$:
\begin{enumerate}
    \item $g=(x,y,z)$ parameterized by $0<y<x$ and\\ $z=\frac{16 x^2+\sqrt{3} \sqrt{80 x^4+352 x^3 y+304 x^2 y^2-64 x y^3+3 y^4}+32 x y-3 y^2}{4 x}$.
    \item $g=(x,y,z)$ parameterized by $x>0$, $y=10x$ and $z=9x$.
\end{enumerate}

\bigskip

\item {\bf Invariant almost complex structure $J_{3}=(+,+,-)$: } if $y=2x$ then $(J_3,g)\in \mathcal{W}_{1}\oplus \mathcal{W}_{2}$, if $x=y=z$ then  $(J_3,g)\in\mathcal{W}_{1}\oplus \mathcal{W}_{3}$, otherwise $(J_3,g)\in \mathcal{W}_{1}\oplus \mathcal{W}_{2}\oplus \mathcal{W}_{3}$. We have
$$\|(dF)^{-}\|^{2}=\frac{4 (x+y+z)^2}{x y z}, \quad \|N^{0}\|^{2}=\frac{128 \left(x^2-x (y+z)+y^2-y z+z^2\right)}{3 x y z}$$, $$\|(dF)^{+}\|^{2}=\frac{6 (y-2 x)^2}{x^2 y}.$$
If $(J_3,g)\in \mathcal{W}_{1}\oplus \mathcal{W}_{3}$ there isn't solution for the equation $2s_1-s=0$. For $(J_2,g)\in \mathcal{W}_{1}\oplus \mathcal{W}_{2}$ we have the metric given by $y=2x$, $z=(2\sqrt{22}+9)x$, $x>0$ is a solution for the equation $2s_1-s=0$. If $(J_3,g)\in \mathcal{W}_{1}\oplus \mathcal{W}_{2}\oplus \mathcal{W}_{3}$ then we have the families of solutions of $2s_1-s=0$:
\begin{enumerate}
    \item $g=(x,y,z)$ parameterized by $z=\frac{\sqrt{-16 x^4+96 x^3 y+560 x^2 y^2-144 x y^3+9 y^4}+24 x y-3 y^2}{4 x}$,
for $x,y>0$ and  $3 x-2 \sqrt{2} \sqrt{x^2}\leq y\leq 2 \sqrt{2} \sqrt{x^2}+3 x$ 
    \item  $g=(x,y,z)$ parameterized by $x=1$, $y=1$ and $z=\frac{1}{4}(21+\sqrt{505})$.
\end{enumerate}

\end{itemize}

\end{document}